\documentclass[11pt,letterpaper]{article}
\usepackage[margin=1in]{geometry}
\usepackage{amsmath}
\usepackage{amssymb}
\usepackage{amsthm}
\usepackage{forloop}
\usepackage{nicefrac}
\usepackage{mathtools}
\usepackage[table]{xcolor}
\usepackage{url}
\usepackage{bm}
\usepackage{float}
\usepackage{hyperref}
\hypersetup{
    colorlinks,
    linkcolor={blue!100!black!70},
    citecolor={blue!75!black},
    urlcolor={blue!75!black}
}
\usepackage[shortlabels]{enumitem}
\usepackage[nameinlink,capitalise]{cleveref}
\usepackage{cancel}
\usepackage{framed}
\usepackage{booktabs}
\usepackage{makecell}

\newcommand{\defcal} [1]{\expandafter\newcommand\csname cal#1\endcsname{{\mathcal #1}}}
\newcommand{\defsf} [1]{\expandafter\newcommand\csname sf#1\endcsname{{\mathsf #1}}}
\newcommand{\defbf} [1]{\expandafter\newcommand\csname bf#1\endcsname{{\mathbf #1}}}
\newcommand{\defbb} [1]{\expandafter\newcommand\csname bb#1\endcsname{{\mathbb{#1}}}}
\newcommand{\deffrak} [1]{\expandafter\newcommand\csname frak#1\endcsname{{\mathfrak{#1}}}}

\newcounter{ct}
\forLoop{1}{26}{ct}{
    \edef\letter{\Alph{ct}}
    \expandafter\defcal\letter
    \expandafter\defsf\letter
    \expandafter\defbf\letter
    \expandafter\defbb\letter
    \expandafter\deffrak\letter
}
\forLoop{1}{26}{ct}{
    \edef\letter{\alph{ct}}
    \expandafter\defcal\letter
    \expandafter\defsf\letter
    \expandafter\defbf\letter
    \expandafter\defbb\letter
    \expandafter\deffrak\letter
}

\newcommand{\rmd}{\mathrm{d}}
\newcommand{\rmI}{\mathrm{I}}

\newcommand{\argmin}{\mathop{\mathrm{argmin}}}
\newcommand{\argmax}{\mathop{\mathrm{argmax}}}

\newcommand{\numberthis}{\addtocounter{equation}{1}\tag{\theequation}}

\newcommand{\KLdist}{\mathrm{d}_{\mathrm{KL}}}

\newcommand{\subscript}[2]{$#1 _ #2$}
\newlist{assumplist}{enumerate}{1}
\setlist[assumplist]{label=\subscript{\textbf{\textsf{A}}}{\textsf{{\arabic*}}},leftmargin=*, itemsep=0pt}

\makeatletter
\newcommand*{\algotitle}[2]{%
  \stepcounter{algocf}%
  \hypertarget{algocf.title.\theHalgocf}{}%
  \NR@gettitle{#1}%
  \label{#2}%
  \addtocounter{algocf}{-1}%
}
\makeatother

\usepackage[round]{natbib}

\allowdisplaybreaks

\newtheorem{lemma}{Lemma}
\newtheorem{theorem}{Theorem}
\newtheorem{proposition}{Proposition}
\newtheorem{corollary}{Corollary}

\newtheorem{fact}{Fact}

\makeatletter
\def\@maketitle{%
  \newpage
  \begin{center}%
  \let \footnote \thanks
    {\Large \bf \@title \par}%
  \end{center}%
  \par
  \vskip 0.5em
}
\makeatother

\newcommand{\opttrans}{\mathrm{OT}}
\newcommand{\allcoups}[2]{\Pi(#1, #2)}

\title{Designing Algorithms for Entropic Optimal Transport\\from an Optimisation Perspective}

\begin{document}
\maketitle

\begin{center}
{\large
\begin{tabular}{ccc}
    \makecell{Vishwak Srinivasan\textsuperscript{*} \\
    {\normalsize Department of EECS, MIT} \\
    {\normalsize\texttt{vishwaks@mit.edu}}} & 
    & \makecell{Qijia Jiang\textsuperscript{*} \\ 
    {\normalsize Department of Statistics, UC Davis} \\
    {\normalsize\texttt{qjang@ucdavis.edu}}}
\end{tabular}
\vskip0.3em
}
\end{center}

\begin{abstract}
In this work, we develop a collection of novel methods for the entropic-regularised optimal transport problem, which are inspired by existing mirror descent interpretations of the Sinkhorn algorithm used for solving this problem.
These are fundamentally proposed from an optimisation perspective: either based on the associated semi-dual problem, or based on solving a non-convex constrained problem over subset of joint distributions.
This optimisation viewpoint results in non-asymptotic rates of convergence for the proposed methods under minimal assumptions on the problem structure.
We also propose a momentum-equipped method with provable accelerated guarantees through this viewpoint, akin to those in the Euclidean setting.
The broader framework we develop based on optimisation over the joint distributions also finds an analogue in the dynamical Schr\"{o}dinger bridge problem.
\end{abstract}

\def\thefootnote{*}\footnotetext{Both authors contributed equally.}

\section{Introduction}
\label{sec:intro}
Given two probability distributions \(\mu\) and \(\nu\) over \(\calX, \calY \subseteq \bbR^{d}\) respectively, the optimal transport (OT) problem concerns finding an \emph{optimal} map that transforms samples from one to another.
The OT problem was originally proposed by Gaspard Monge in the 1780s to address the problem of finding a method to transport resources between a collection of sources and sinks, and was rediscovered in the early 1900s by Hitchcock, Kantorovich, and Koopmans with applications in designing transportation systems, coinciding with the birth of linear programming.
Recent advances in computing resources has renewed interest in the OT problem, both in the design of approximate methods for this problem suited for large-scale settings \citep{peyre2019computational}, and in the development of a theoretical understanding of its properties \citep{villani2003topics,santambrogio2015optimal}.

The ``optimality'' in the OT problem is defined in terms of a cost function \(c : \calX \times \calY \to \bbR\).
The optimal value of the problem results in a notion of discrepancy between \(\mu\) and \(\nu\) that complements information-theoretic discrepancy measures like the total variation distance or the Kullback-Leibler (KL) divergence.
Formally, let \(\Pi(\mu, \nu)\) be the set of all joint distributions whose marginals are \(\mu\) and \(\nu\).
The Kantorovich formulation of the OT problem is given by the following program:
\begin{equation}
\label{eqn:OT-prob-primal}
    \inf_{\pi \in \allcoups{\mu}{\nu}} \iint c(x, y) \rmd\pi(x, y) =: \opttrans(\mu, \nu; c)~.
\end{equation}

A notable case is when \(c(x, y) = \|x - y\|_{p}^{p}\) where \(\|\cdot\|_{p}\) is the \(\ell_{p}\)-norm, and this corresponds to the \(p\)-Wasserstein distance between \(\mu\) and \(\nu\) raised to the \(p^{th}\) power.
The space of probability measures endowed with the \(p\)-Wasserstein distance takes on a rich Riemannian structure with implications in developing the metric theory of gradient flows \citep{ambrosio2008gradient}.
Wasserstein distances have also found applications in image processing and operations research, which has motivated the design of efficient methods to compute these distances.
In modern machine learning and statistics however, methods based on solving the exact OT problem have not seen widespread use owing to both computational and statistical reasons resulting primarily from the high-dimensional nature of the problems involved.

These bottlenecks can surprisingly be alleviated by adding an entropy regularisation to the OT problem, referred to as the entropic optimal transport (eOT) problem.
More precisely, for a regularisation parameter \(\varepsilon > 0\), this is defined as
\begin{equation}
\label{eqn:eOT-prob-primal}
    \inf_{\pi \in \allcoups{\mu}{\nu}} \iint c(x, y) \rmd\pi(x, y) + \varepsilon \cdot \KLdist(\pi \| \mu \otimes \nu) =: \opttrans_{\varepsilon}(\mu, \nu; c)~.
\end{equation}
Above, \(\KLdist(\pi \| \mu \otimes \nu)\) is the KL divergence between \(\pi\) and the product distribution \(\mu \otimes \nu\).
A popular algorithm for solving this problem is the \emph{Sinkhorn} algorithm \citep{sinkhorn1967concerning}, and this was recently popularised by \citet{cuturi2013sinkhorn} by demonstrating it as a viable solution for solving the eOT problem over large datasets, and by extension for solving the OT problem approximately.
We refer the reader to \citet[Remark 4.5]{peyre2019computational} for a historical view of the method.
Although \citet{sinkhorn1967concerning} proved that the method converges asymptotically, the first known non-asymptotic rate of convergence was given by \citet{franklin1989scaling}.
They do so by viewing the Sinkhorn algorithm as a matrix scaling method and leverage Hilbert's projective metric in conjunction with Birkhoff's theorem.
Other interpretations of the Sinkhorn algorithm have led to both asymptotic and non-asymptotic guarantees; see \cite{ruschendorf1995convergence,marino2020optimal,carlier2022linear} for instance.
We highlight that \(\opttrans_{\varepsilon}(\mu, \nu; c)\) is \emph{biased} relative to \(\opttrans(\mu, \nu; c)\) as \(\opttrans_{\varepsilon}(\mu, \mu; c) \neq 0\).
To address this, a debiasing strategy is proposed in \cite{feydy2019interpolating} which results in the \emph{Sinkhorn divergence} which is shown to metrize convergence in law.
In addition to the computational benefits of the Sinkhorn algorithm, the sample complexity to estimate the Sinkhorn divergence using samples from \(\mu\) and \(\nu\) scales much better than \(\opttrans(\mu, \nu; c)\) for a variety of costs \citep{genevay2019sample,mena2019statistical,chizat2020faster}.

Of relevance to this paper is a more recent interpretation of the Sinkhorn algorithm as an \emph{infinite-dimensional optimisation procedure}, which originated with the mirror descent interpretation of the Sinkhorn algorithm in \citet{mishchenko2019sinkhorn} for discrete spaces \(\calX\) and \(\calY\).
This interpretation has been instrumental in obtaining assumption-free guarantees for the Sinkhorn algorithm and has led to a growing body of literature since its inception \citep{mensch2020online,leger2021gradient,aubin2022mirror,deb2023wasserstein,karimi2024sinkhorn}.

\paragraph{Summary of contributions}
In this work, we draw inspiration from this refreshing viewpoint, and design new methods for the eOT problem with provable guarantees.
Notably, these guarantees do not place any specific assumptions on the domains \(\calX\) and \(\calY\), or on the marginal distributions \(\mu\) and \(\nu\).
A key mathematical object that we consider in ths design is the \emph{semi-dual formulation} of the eOT problem which we describe in more detail in \cref{sec:new-class}.
The semi-dual problem is a concave unconstrained program, which motivates the use of a steepest ascent procedure and with the Sinkhorn algorithm leads to an abstract class of methods.
Roughly speaking, these methods can be seen as iteratively minimising the discrepancy between \(\nu\) and the \(\calY\)-marginal of a joint distribution while ensuring the joint distribution maintains a certain form.

We show that these class of methods can be interpreted in a variety of ways as described in \cref{sec:measure_interpretation}.
These interpretations come in handy when analysing a subclass of methods that can be seen as minimising the squared maximum mean discrepancy \citep{gretton2006kernel}.
We find that the \(\calY\)-marginal of the sequence of joint distributions resulting from these methods converges to \(\nu\) at a rate that scales as \(\frac{1}{N}\) where \(N\) is the number of iterations, and consequently leads to an optimal coupling for the eOT problem due to the form of the joint distribution that is maintained.

We next show how we can use principles in finite-dimensional optimisation to adapt to growth conditions of the semi-dual to design other kinds of steepest descent methods, extending the class of methods described a priori for solving the eOT problem in \cref{sec:extensions:steep-ascent}.
These also lead to methods that provably optimise the semi-dual (and consequently solving the eOT problem) and converge at a rate that scales \(\frac{1}{N}\) under minimal assumptions.
We find that one of these methods can be accelerated and converge at a faster rate that scales as \(\frac{1}{N^{2}}\) under the same assumptions.
We finally discuss a path-space generalisation of the class of methods in \cref{sec:new-class} for solving the dynamical Schr\"{o}dinger bridge problem in \cref{sec:extensions:SB}.

\section{Background}
\label{sec:background}
\paragraph{Notation}
For a set \(\calZ\), the set of probability measures over \(\calZ\) is denoted by \(\calP(\calZ)\).
Given a distribution \(\pi \in \calP(\calX \times \calY)\), \(\pi_{\calX}\) and \(\pi_{\calY}\) are the \(\calX\)-marginal and \(\calY\)-marginal respectively i.e., for \(A \subseteq \calX, B \subseteq \calY\),
\begin{equation*}
    \pi_{\calX}(A) = \int_{A \times \calY}\rmd\pi(x,y) ~; \quad \pi_{\calY}(B) = \int_{\calX \times B}\rmd \pi(x, y)~.
\end{equation*}
Given two distributions \(\mu \in \calP(\calX), \nu \in \calP(\calY)\), we say that \(\pi \in \calP(\calX \times \calY)\) is a coupling of \(\mu\) and \(\nu\) if \(\pi_{\calX} = \mu\) and \(\pi_{\calY} = \nu\).
For \(\rho \in \calP(\calZ)\), we use \(\rmd\rho\) to represent its density.
For convenience, if \(\rho\) has a density w.r.t. the Lebesgue measure, we use \(\rho\) to also denote its density depending on the setting.

Let \(\rho, \rho' \in \calP(\calZ)\) be probability measures such that \(\rho\) is absolutely continuous w.r.t. \(\rho'\).
The KL divergence between \(\rho, \rho'\) is denoted by \(\KLdist(\rho \| \rho') := \int_{\calZ} \rmd \rho \log \frac{\rmd \rho}{\rmd \rho'}\).
For a functional \(\calF : \calP(\calZ) \to \bbR\), we call \(\delta \calF(\rho)\) its first variation at \(\rho \in \calP(\calZ)\), and this is the function (up to additive constants) that satisfies \cite[Def. 7.12]{santambrogio2015optimal}
\begin{equation*}
    \langle \delta\calF(\rho), \chi\rangle = \int_{\calZ} \delta \calF(\rho)(z)\rmd\chi(z) = \lim_{h \to 0} \frac{\calF(\rho + h \cdot \chi) - \calF(\rho)}{h} \qquad \forall ~\chi \text{ such that} \int \rmd \chi(z) = 0~.
\end{equation*}

For a measurable function \(f : \calZ \to \bbR\), its \(L^{p}\)-norm w.r.t. \(\rho\) is denoted by \(\|f\|_{L^{p}(\rho)}\), which is defined as \(\left(\int_{\calZ} |f(z)|^{p}\rmd \rho(z)\right)^{\nicefrac{1}{p}}\).
When \(\rho\) is replaced with \(\calZ\) as \(\|f\|_{L^{p}(\calZ)}\), then this is understood to be the \(L^{p}\)-norm of \(f\) w.r.t. the Lebesgue measure of \(\calZ\).
For another measurable function \(g : \calZ \to \bbR\), the \(L^{2}(\rho)\) inner product is defined as \(\langle f, g\rangle_{L^{2}(\rho)} = \int_{\calZ} f(z)g(z) \rmd \rho(z)\).
As a special case, when the subscript in the norm and inner product are omitted as in \(\|\cdot\|\) and \(\langle\cdot, \cdot\rangle\), then these correspond to the \(L^{2}(\calZ)\)-norm and inner product respectively.
Following the notation in \citet[Chap. 3]{rudin1987real}, we use \(L^{p}\) to also denote the space of measurable functions whose \(L^{p}\) norm is finite.

\subsection{The entropic optimal transport problem}

We recall the definition of the eOT problem from \cref{eqn:eOT-prob-primal} below. For a lower-semi-continuous cost cost function \(c : \calX \times \calY \to \bbR\), the primal formulation of the entropic optimal transport (abbrev. eOT) problem is given by
\begin{equation*}
    \inf_{\pi \in \allcoups{\mu}{\nu}} \iint c(x, y) \rmd\pi(x, y) + \varepsilon \cdot \KLdist(\pi \| \mu \otimes \nu) =: \opttrans_{\varepsilon}(\mu, \nu; c)~.
\end{equation*}
When \(\varepsilon = 0\), this corresponds to the (canonical) optimal transport problem.
On the other hand, it can be seen that as \(\varepsilon \to \infty\), the solution to \(\opttrans_{\varepsilon}(\mu, \nu; c)\) tends to the product measure \(\mu \otimes \nu\).
The eOT problem is also directly related to the \emph{static Schr\"{o}dinger bridge problem} \cite[Def. 2.2]{leonard2014survey} denoted by \(\mathrm{SB}(\mu, \nu; \pi^{\textsf{ref}})\) for the reference measure \(\pi^{\textsf{ref}}\) whose density is \(\rmd \pi^{\textsf{ref}} \propto \exp\left(-\nicefrac{c}{\varepsilon}\right)\rmd (\mu \otimes \nu)\) and marginals \(\mu\) and \(\nu\).
Then, direct calculation gives
\begin{equation}
\label{eqn:static_SB}
    \opttrans_{\varepsilon}(\mu, \nu; c) = \varepsilon \cdot \Big\{\underbrace{\inf_{\pi \in \Pi(\mu, \nu)} \KLdist(\pi \| \pi^{\textsf{ref}})}_{\mathrm{SB}(\mu, \nu; \pi^{\textsf{ref}})} - \log Z_{\textsf{ref}}\Big\}~; ~~ Z_{\textsf{ref}} = \iint \exp\left(-\frac{c(x, y)}{\varepsilon}\right)\rmd\mu(x)\rmd\nu(y)~.
\end{equation}
The above problem is a (strictly) convex minimisation problem since \(\Pi(\mu, \nu)\) is convex, and \(\pi \mapsto \KLdist(\pi \| \pi^{\textsf{ref}})\) is strictly convex.
Moreover, under certain regularity conditions, the eOT problem admits a unique solution \(\pi^{\star} \in \Pi(\mu, \nu)\) \citep{leonard2014survey,nutz2022entropic} of the form
\begin{equation}
\label{eqn:eOT-opt-form}
    \rmd\pi^{\star}(x, y) = \exp\left(\phi^{\star}(y) - \psi^{\star}(x) - \frac{c(x, y)}{\varepsilon}\right)\rmd \mu(x)\rmd \nu(y) 
\end{equation}
where \(\psi^{\star}\) and \(\phi^{\star}\) are called \emph{Schr\"{o}dinger potentials}, and the optimal value of the problem is
\begin{equation*}
    \opttrans_{\varepsilon}(\mu, \nu; c) = \varepsilon \cdot \left(\int \phi^{\star}(y)\rmd \nu(y) - \int \psi^{\star}(x)\rmd\mu(x)\right)~.
\end{equation*}
While the solution \(\pi^{\star}\) is unique, the Schr\"{o}dinger potentials are unique only up to constants, that is, if \(\psi^{\star}\) and \(\phi^{\star}\) are Schr\"{o}dinger potentials, then \(\psi^{\star} + \beta\) and \(\phi^{\star} + \beta\) are also Schr\"{o}dinger potentials for any constant \(\beta \in \bbR\).
The eOT problem has a dual formulation with zero duality gap, and is \emph{unconstrained}:
\begin{align*}
    \opttrans_{\varepsilon}(\mu, \nu; c) &= \varepsilon \cdot \sup \left\{ \int \phi(y)\rmd \nu(y) -\int \psi(x)\rmd \mu(x) \right.\\
    &\qquad \qquad \qquad \left.- \log \iint \exp\left(\phi(y) -\psi(x) - \frac{c(x, y)}{\varepsilon}\right) \rmd\mu(x)\rmd \nu(y) \right\}~.\numberthis\label{eqn:eOT-prob-dual}
\end{align*}
Any solution of the dual problem above corresponds to a pair of Schr\"{o}dinger potentials and vice versa \cite[Thm. 3.2]{nutz2021introduction}.
Note that \(\opttrans_{\varepsilon}(\mu, \nu; c) = \varepsilon \cdot \opttrans_{1}(\mu, \nu; \nicefrac{c}{\varepsilon})\), and hence without loss of generality, we focus on \(\opttrans_{1}(\mu, \nu; \nicefrac{c}{\varepsilon})\) in the rest of this work.
We denote the objective in the dual form of \(\opttrans_{1}(\mu, \nu; \nicefrac{c}{\varepsilon})\) by \(D(\psi, \phi)\) and use \(\pi^{\star}\) to denote the (primal) solution of \(\opttrans_{1}(\mu, \nu; \nicefrac{c}{\varepsilon})\).

\subsection{Related work}
Traditional analyses view the Sinkhorn algorithm as either alternating projection on the two marginals $\mu,\nu$ or block maximization on the two dual potentials $\psi, \phi$.
These render a linear convergence with a contraction rate of the form \(1 - e^{-\nicefrac{\|c\|_{\infty}}{\varepsilon}}\).
An important limitation of this analysis is that the rate becomes \emph{exponentially slower} with growing $\|c\|_\infty$ or decreasing $\varepsilon$.
Recently, several analyses have focused on the Sinkhorn algorithm in the setting where \(\calX\) and \(\calY\) are discrete spaces \citep{altschuler2017near, lin2022efficiency, dvurechensky2018computational}.
More relevant to us is where \(\calX\) and \(\calY\) are continuous spaces where many probabilistic approaches have been taken for analyzing the Sinkhorn algorithm.
Most recently, \citet{chiarini2024semiconcavity} leverage the stability of optimal plans with respect to the marginals to obtain exponential convergence with unbounded cost for all $\varepsilon>0$, albeit under various sets of conditions on the marginals.
This relaxes the assumptions made in \citet{chizat2024sharper} for semi-concave bounded costs while still maintaining a contraction rate that only deteriorates \emph{polynomially} in $\varepsilon$.
We refer the reader to \citet[Sec. 1.5]{chiarini2024semiconcavity} for a more comprehensive literature review for analyses of the Sinkhorn algorithm.
In summary, the most recent analyses place assumptions on the growth of the cost, decay of the tails of \(\mu, \nu\) and/or log-concavity, to obtain exponential convergence guarantees.

In contrast, the advantage of taking the optimisation route i.e., viewing the Sinkhorn algorithm as performing infinite-dimensional mirror descent \citep{leger2021gradient,aubin2022mirror, karimi2024sinkhorn} is that it provides a guarantee under \emph{minimal} assumptions.
From non-asymptotic guarantee standpoint, these aforementioned works furnish a discrete-time iteration complexity that scales as \(\nicefrac{1}{N\varepsilon}\).
If the costs are additionally assumed to be bounded, then \citet{aubin2022mirror} recover a contractive rate reminiscent of the classical Hilbert analysis.
In this work, we achieve the similar rates while significantly expanding the scope of algorithm design and shed more light on the eOT problem, unifying both the primal and dual perspectives.
\cite{mensch2020online} gives another mirror descent interpretation of Sinkhorn but the change of variable results in a non-convex objective which is hard to prove convergence for.
There has also been interest in designing alternative algorithms for the eOT problem, among them \cite{conforti2023projected} that designs Wasserstein gradient flow dynamics over the submanifold of $\Pi(\mu,\nu)$ which borrow tools from SDEs and PDEs in its analysis. 

\section{A new class of methods for solving the eOT problem}
\label{sec:new-class}
In this section, we propose a class of methods for solving the eOT problem, termed \ref{eqn:phi-match-update}.
As we explain later in this section, this new class of updates is derived by identifying commonalities between the Sinkhorn algorithm and another method that is based on the \emph{semi-dual formulation} associated with the dual eOT problem in \cref{eqn:eOT-prob-dual}.
We begin by first introducing this semi-dual problem in \cref{sec:semi-dual} and present this class of methods in \cref{sec:new-class:generalise}.

\subsection{The semi-dual problem in eOT}
\label{sec:semi-dual}
The semi-dual problem was originally discussed in \cite{genevay2016stochastic}, specifically to motivate the use of stochastic algorithms to solve the eOT problem in the setting where \(\calX\) and \(\calY\) are discrete spaces.
Later work by \cite{cuturi2018semi} studies this in more detail, while still primarily focusing on the discrete space setting.

We first define the following operations, in the notation of \citet{leger2021gradient}.
Let \(\phi \in L^{1}(\nu)\) and \(\psi \in L^{1}(\mu)\).
Define
\begin{equation*}
    \phi^{+}(x) := \log\int_{\calY}\exp\left(\phi(y) - \frac{c(x, y)}{\varepsilon}\right)\rmd \nu(y)~; ~~ \psi^{-}(y) := -\log \int_{\calX} \exp\left(\psi(x) + \frac{c(x, y)}{\varepsilon}\right)\rmd \mu(x)~.
\end{equation*}

Importantly, from \cref{eqn:eOT-opt-form} we see that any pair of Schr\"{o}dinger potentials \((\phi^{\star}, \psi^{\star})\) corresponding to \(\opttrans_{1}(\mu, \nu; \nicefrac{c}{\varepsilon})\) satisfies \(\psi^{\star} = (\phi^{\star})^{+}\) and \(\phi^{\star} = (\psi^{\star})^{-}\); this can be inferred by noting that \(\pi^{\star}_{\calY} = \nu\) and \(\pi^{\star}_{\calX} = \mu\).
Therefore, instead of solving the dual problem \cref{eqn:eOT-prob-dual} in two variables \(\phi, \psi\), it would be sufficient to solve one of
\begin{equation*}
    \sup_{\phi \in L^{1}(\nu)} D(\phi^{+}, \phi) \qquad \text{or} \qquad \sup_{\psi \in L^{1}(\mu)} D(\psi, \psi^{-})~.
\end{equation*}
Without loss of generality, we work with the maximisation problem over \(\phi\), and this is referred to as the \emph{semi-dual problem} for eOT.
Additionally, we note that for any \(\phi \in L^{1}(\nu)\), the objective \(J\) of the semi-dual satisfies
\begin{equation}
\label{eqn:J-def}
    J(\phi) := D(\phi^{+}, \phi) = \sup_{\psi \in L^{1}(\mu)} D(\psi, \phi) = \int_{\calY} \phi(y)\rmd \nu(y) - \int_{\calX} \phi^{+}(x)\rmd \mu(x)\, .
\end{equation}
The fact that \(D(\phi^{+}, \phi) = \sup_{\psi \in L^{1}(\mu)} D(\psi, \phi)\) can be derived by using the fact that \(\psi \mapsto D(\psi, \phi)\) is concave and its first variation at \(\phi^{+}\) is \(0\).
This leads to viewing the semi-dual problem for eOT as explicitly eliminating \(\psi\) via a partial maximisation of \(D(\psi, \phi)\) in the dual problem (\cref{eqn:eOT-prob-dual}).
%The objective \(J\) is referred to as the \emph{semi-dual}.
To translate a dual potential \(\phi \in L^{1}(\nu)\) to a joint distribution over \(\calX \times \calY\), define \(\pi(\phi, \phi^{+})\) with density w.r.t. \(\mu \otimes \nu\) as
\begin{equation}
\label{eqn:primal_pi}
    \rmd\pi(\phi, \phi^{+})(x, y) = \exp\left(\phi(y) - \phi^{+}(x) - \frac{c(x,y)}{\varepsilon}\right)\rmd\nu(y)\rmd\mu(x)~.
\end{equation}
This is a valid probability density function over \(\calX \times \calY\) as evidenced by the fact that its \(\calX\)-marginal is always \(\mu\).
Recall that this joint distribution with \(\phi \leftarrow \phi^{\star}\) in \cref{eqn:primal_pi} corresponds to the unique solution of the eOT problem \cref{eqn:eOT-opt-form} by virtue of \((\phi^{\star})^{+} = \psi^{\star}\).

\subsubsection{Properties of the semi-dual \(J\)}

We henceforth assume that \(\mu\) and \(\nu\) have densities w.r.t. the Lebesgue measure.
In this subsection, we state properties of the semi-dual \(J\) that underlie the methods that we propose and study in this section.
As a prelude, we state two general observations about the semi-dual \(J\).

\begin{fact}[Shift-invariance]
\label{rmk:J_shift_invariance}
The semi-dual is invariant to additive perturbations of its argument.
Formally, for any \(C \in \bbR\) and \(\phi \in L^{1}(\nu)\), \(J(\phi + C \cdot \bm{1}) = J(\phi)\) where \(\bm{1} : x \mapsto 1\).
This is due to the fact that \((\phi + C \cdot \bm{1})^{+} = \phi^{+} + C \cdot \bm{1}\).
\end{fact}
\begin{fact}[First variation]
\label{fct:first-var}
The first variation \(\delta J\) of the semi-dual \(J\) can be succinctly expressed in terms of the marginal \(\pi(\phi, \phi^{+})_{\calY}\) \cite[Lem. 1]{leger2021gradient}: for any \(\phi \in L^{1}(\nu)\),
\begin{equation}
\label{eqn:first-var-J}
    \delta J(\phi)(y) = \nu(y) - \pi(\phi, \phi^{+})_{\calY}(y) = \nu(y) - \int_{\calX} \exp\left(\phi(y) - \phi^{+}(x) - \frac{c(x, y)}{\varepsilon}\right)\nu(y)\mu(x)\rmd x~.
\end{equation}
\end{fact}

The following two properties play a crucial role in our analysis. The first is that the Bregman divergence of \(J\) is non-positive -- in other words \(J\) is a concave functional.

\begin{lemma}
\label{lem:concavity-of-J}
Let \(\phi, \overline{\phi} \in L^{1}(\nu)\). Then, \(J(\overline{\phi}) - J(\phi) -\left\langle \delta J(\phi), \overline{\phi} - \phi \right\rangle \leq 0\).
\end{lemma}

The second property is a lower bound on the Bregman divergence, yielding the regularity condition that the semi-dual \(J\) does not grow arbitrarily.

\begin{lemma}
\label{lem:Linf-smooth}
Let \(\phi, \overline{\phi} \in L^{1}(\nu)\).
Then,
\begin{equation*}
    J(\overline{\phi}) - J(\phi) - \langle \delta J(\phi), \overline{\phi} - \phi\rangle \geq -\frac{\|\overline{\phi} - \phi\|_{L^{\infty}(\calY)}^{2}}{2}~.
\end{equation*}
\end{lemma}

\subsection{From semi-dual gradient ascent to a new class of methods for eOT}
\label{sec:new-class:generalise}

\cref{lem:concavity-of-J} and \cref{fct:first-var} implies that one could ostensibly use a gradient ascent-like procedure to find a maximiser of the semi-dual \(J\) and consequently solve the eOT problem to within a desired tolerance.
More precisely, from an initialisation \(\phi^{0} \in L^{1}(\nu)\), we can obtain a sequence of iterates \(\{\phi^{n}\}_{n \geq 1}\) based on the recursion
\begin{equation*}
    \phi^{n + 1} = \sfM^{\textsf{SGA}}(\phi^{n}; \eta) := \phi^n + \eta \cdot \delta J(\phi^{n})~.\tag{\textsf{SGA}}\label{eqn:sga-update}
\end{equation*}

The update \(\sfM^{\textsf{SGA}}\) was previously considered by \cite{genevay2016stochastic} for discrete spaces \(\calX\) and \(\calY\), where \(\phi\) can be represented as a finite-dimensional vector.
In this setting, \cref{lem:Linf-smooth} implies a standard notion of smoothness for the semi-dual \(J\) \cite[Chap. 2]{nesterov2018lectures} by the monotonicity of finite-dimensional norms.
This consequently results in an \emph{assumption-free} non-asymptotic convergence guarantee for \ref{eqn:sga-update} with \(\eta < 2\).

When generalising to continuous spaces, a temporary setback towards establishing such rates of convergence for this update is that \(\|\phi\|_{L^{\infty}(\calY)} \leq \|\phi\|_{L^{2}(\calY)}\) is not generally true, thus leading to an ``incompatibility''.
We rectify this by instead adopting an alternate perspective on \ref{eqn:sga-update} as minimising the ``discrepancy'' between the \(\calY\)-marginal of \(\pi(\phi, \phi^{+})\) and \(\nu\).
From the form of \(\delta J(\phi)\) in \cref{eqn:first-var-J}, we can rewrite the update \ref{eqn:sga-update} as
\begin{equation}
\label{eqn:sga_explicit}
    \sfM^{\textsf{SGA}}(\phi; \eta) = \phi + \eta \cdot \left(\nu - \pi(\phi, \phi^{+})_{\calY}\right)~.
\end{equation}
Note that a fixed point \(\phi^{\star}\) of this update satisfies \(\pi(\phi^{\star}, (\phi^{\star})^{+})_{\calY} = \nu\).
This also corresponds to a maximiser of \(J\) since \(\pi(\phi^{\star}, (\phi^{\star})^{+})_{\calY} = \nu\) which is equivalent to \(\delta J(\phi^{\star}) = 0\).

Related to \cref{eqn:sga_explicit}, the Sinkhorn algorithm corresponds to an update map \(\sfM^{\textsf{Sinkhorn}}\) \citep[Lem. 2]{karimi2024sinkhorn} that is defined as
\begin{equation*}
    \sfM^{\textsf{Sinkhorn}}(\phi) := \phi - \left(\log \frac{\pi(\phi, \phi^{+})_{\calY}}{\nu}\right)~.\tag{\textsf{Sinkhorn}}\label{eqn:sinkhorn-update}
\end{equation*}
Through a little observation, it can be seen that both \(\sfM^{\textsf{SGA}}\) and \(\sfM^{\textsf{Sinkhorn}}\) are instantiations of a more general class of updates described below.
Let \(\Phi : L^{1}(\calY) \to L^{\infty}(\calY)\) be an operator that returns a positive function.
In this setting, it takes a density function (\(L^{1}(\calY)\)) and returns a function that is not necessarily a density (\(L^{\infty}(\calY)\)).
Then, we define the update \(\sfM^{\Phi\textsf{-match}}\) as
\begin{equation*}
    \sfM^{\Phi\textsf{-match}}(\phi; \eta) = \phi - \eta \cdot \left(\log \Phi(\pi(\phi, \phi^{+})_{\calY}) - \log \Phi(\nu) \right)~.\tag{\(\Phi\)\textsf{-match}}\label{eqn:phi-match-update}
\end{equation*}
From this, we see that (1) when \(\Phi(f) : f \mapsto e^{f}\), this recovers \ref{eqn:sga-update}; and (2) when \(\Phi(f) : f \mapsto f\), this recovers \ref{eqn:sinkhorn-update} with a general step size \(\eta\) (also called \(\eta\)-Sinkhorn in \cite{karimi2024sinkhorn}).
In the next section, we provide different interpretations of \ref{eqn:phi-match-update} in the primal space, i.e., over the space of distributions \(\pi\in\calP(\calX \times \calY)\) as opposed to dual potentials \(\phi\in L^{1}(\nu)\) as we do here.
These interpretations not only lend to rates of convergence for \ref{eqn:sga-update}, but also a related collection of methods that are instances of the general \ref{eqn:phi-match-update} framework.

\section{Interpretations of \textsf{\(\Phi\)-match}}
\label{sec:measure_interpretation}
\newcommand{\projX}{\mathsf{project}_{\calX, \mu}}
\newcommand{\projY}{\mathsf{project}_{\calY, \nu}}
\newcommand{\rootKL}{\mathsf{root}_{\calX, \mu}}

The interpretations of \ref{eqn:phi-match-update} that we discuss here are motivated by recent work in understanding the Sinkhorn algorithm (\ref{eqn:sinkhorn-update}) \citep{aubin2022mirror,karimi2024sinkhorn}.
In essence, these prior works view the Sinkhorn algorithm as minimising \(\KLdist(\pi_{\calY}, \nu)\) over a subset \(\calQ\) of joint distributions over \(\calX \times \calY\).
This is defined as
\begin{equation}
\label{eqn:joint-constraint}
    \calQ := \left\{\pi : \exists~ \phi \in L^{1}(\nu) \text{ such that } \pi(x, y) = \exp\left(\phi(y) - \phi^{+}(x) - \frac{c(x, y)}{\varepsilon}\right) \mu(x)\nu(y)\right\}~.
\end{equation}
While \(\pi^{\star}\) belongs to \(\calQ\), this constrained set is \emph{not} a convex subset of \(\calP(\calX \times \calY)\) owing to the factorisation structure.
Despite this, an intriguing observation about the Sinkhorn algorithm is that it ensures the iterates lie in this set \(\calQ\).
Here, we show that \ref{eqn:phi-match-update} also operates in the same way while minimising an objective that is not the KL divergence but instead a discrepancy which depends on \(\Phi\).
We specifically show that \ref{eqn:phi-match-update} can be interpreted in the following two ways: (1) as an alternating projection scheme and (2) as a local greedy method analogous to gradient descent / mirror descent.
While \ref{eqn:phi-match-update} is derived from the semi-dual, these interpretations do not involve the semi-dual and solely operate in \(\calP(\calX \times \calY)\).
This leads to rates of convergence for \ref{eqn:sga-update} and its kernelised version (\ref{eqn:kernel-sga-update}) that we introduce later on.

\paragraph{\ref{eqn:phi-match-update} as iterative projections on \(\calP(\calX \times \calY)\)}

Consider the following projection operations.
\begin{subequations}
\begin{align}
\projY(\pi; \Phi) &:= \argmin_{\bar{\pi}} \left\{\KLdist(\bar{\pi}\Vert \pi): \bar{\pi}_{\calY} \propto \pi_{\calY} \cdot \frac{\Phi(\nu)}{\Phi(\pi_{\calY})}\right\}, \label{eqn:measure_1}\\
\projX(\pi', \pi; \eta) &:=\argmin_{\bar{\pi}} \left\{\eta \cdot \KLdist(\bar{\pi}\Vert\pi')+(1-\eta) \cdot \KLdist(\bar{\pi}\Vert \pi): \bar{\pi}_{\calX}=\mu\right\}~.
\label{eqn:measure_2} 
\end{align}
\end{subequations}
For a given \(\pi \in \calP(\calX \times \calY)\), \(\projY\) can be seen as correcting the \(\calY\)-marginal of \(\pi\) towards \(\nu\), and the nature of this correction depends on \(\Phi\).
On the other hand, \(\projX\) finds a ``midpoint'' (which depends on the stepsize $\eta \in [0, 1]$) while ensuring that the \(\calX\)-marginal is \(\mu\).
In the following lemma, we show that if \(\pi \in \calQ\) (corresponding to some \(\phi \in L^{1}(\nu))\), then applying the projections \cref{eqn:measure_1,eqn:measure_2} successively is equivalent to updating \(\phi\) using \ref{eqn:phi-match-update}.

\begin{lemma}
\label{lem:KL_primal_projection}
Let \(\phi^{0} \in L^{1}(\nu)\) and let \(\{\phi^{n}\}_{n \geq 1}\) be the sequence of potentials obtained as \(\phi^{n + 1} = \sfM^{\Phi\emph{\textsf{-match}}}(\phi^{n}; \eta)\) for \(\eta in [0, 1]\).
Then, the sequence of distributions \(\{\pi^{n}\}_{n \geq 0}\) where \(\pi^{n} = \pi(\phi^{n}, (\phi^{n})^{+})\) satisfy for every \(n \geq 0\)
\begin{equation*}
    \pi^{n + 1} = \projX(\pi^{n + \nicefrac{1}{2}}, \pi^{n}; \eta) \quad \text{where}~ \pi^{n + \nicefrac{1}{2}} = \projY(\pi^{n}; \Phi)~.
\end{equation*}
\end{lemma}

\paragraph{\ref{eqn:phi-match-update} as a local greedy method}

For a map \(\calF : L^{1}(\calX \times \calY) \to L^{\infty}(\calX \times \calY)\), define
\begin{equation}
\label{eqn:proximal-point}
    \rootKL(\pi; \calF, \eta) := \argmin_{\bar{\pi}} \left\{\langle \calF(\pi),\bar{\pi}-\pi \rangle + \eta^{-1} \cdot \KLdist(\bar{\pi}\Vert \pi) : \bar{\pi}_{\calX} = \mu \right\}~.
\end{equation}
When \(\calF\) is the first variation of a functional that measures the discrepancy between the \(\calY\)-marginal of \(\pi\) and \(\nu\), \(\rootKL(\pi; \calF, \eta)\) can be viewed as minimising a local first-order approximation of this functional, thereby approximately matching the \(\calY\)-marginal, while restricting the $\calX$-marginal to be \(\mu\).
The following lemma states that the successive projections defined in \cref{eqn:measure_1,eqn:measure_2} correspond to a root finding procedure for \(\calF \leftarrow \calV_{\Phi}\) defined as
\begin{equation}
\label{eqn:proximal-point-V}
    \calV_{\Phi}(\pi)(x, y) = \log \Phi(\pi_{\calY})(y) - \log \Phi(\nu)(y)~.
\end{equation}

\begin{lemma}
\label{lem:proximal-point}
Let \(\pi \in \calP(\calX \times \calY)\) be such that \(\pi_{\calX} = \mu\).
Then for \(\eta \in [0, 1]\),
\begin{equation*}
    \projX(\projY(\pi; \Phi), \pi; \eta) = \rootKL(\pi; \calV_{\Phi}, \eta)~. 
\end{equation*}
\end{lemma}

The proofs of \cref{lem:KL_primal_projection,lem:proximal-point} are given in \cref{sec:prf:KL_primal_projection,sec:prf:proximal-point} respectively.
This leads to the following straightforward corollary about \ref{eqn:phi-match-update} in the manner of \cref{lem:KL_primal_projection}.
\begin{corollary}
\label{corr:phi_rewrite}
Let \(\phi^{0} \in L^{1}(\nu)\) and let \(\{\phi^{n}\}_{n \geq 1}\) be the sequence of potentials obtained as \(\phi^{n + 1} = \sfM^{\Phi\emph{\textsf{-match}}}(\phi^{n}; \eta)\) for \(\eta \in [0, 1]\).
Then, the sequence of distributions \(\{\pi^{n}\}_{n \geq 0}\) where \(\pi^{n} = \pi(\phi^{n}, (\phi^{n})^{+})\) satisfy for every \(n \geq 0\)
\begin{equation*}
    \pi^{n + 1} = \rootKL(\pi^{n}; \calV_{\Phi}, \eta)~.
\end{equation*}
\end{corollary}

The above corollary shows that while the updates \ref{eqn:phi-match-update} are themselves expressed in terms of the potentials \(\phi\), the joint distributions are automatically constrained in the set $\calQ$ and solve \cref{eqn:proximal-point} at every iteration.

Recall from earlier that \ref{eqn:phi-match-update} with \(\Phi : f \mapsto f\) and \(\eta = 1\) coincides with \ref{eqn:sinkhorn-update}.
In this setting, \cref{lem:KL_primal_projection} recovers the classical iterative Bregman projection interpretation of the Sinkhorn method \cite[Remark 4.8]{peyre2019computational}.
The map \(\calV_{\Phi}\) in \cref{eqn:proximal-point-V} is \(\log \frac{\pi_{\calY}}{\nu}\) in this case, which corresponds to the first variation of the functional \(\rho \mapsto \KLdist(\rho_{\calY} \| \nu)\).
This is of course, up to an additive constant which cancels out in the inner product in \cref{eqn:proximal-point}.
The equivalence in \cref{corr:phi_rewrite} for the Sinkhorn method was originally derived in \citet[Prop. 5]{aubin2022mirror} and generalised to an arbitrary step size \(\eta \in (0, 1)\) in \citet[Lem. 1]{karimi2024sinkhorn}.
Besides the KL divergence, another popular measure of discrepancy between distributions is the \(\chi^{2}\) divergence.
Setting \(\Phi(f) = \exp\left(\frac{f}{\nu} - 1\right)\), and \(\calV_{\Phi}\) results in \(\calV_{\Phi}\) coinciding with \(\frac{\pi_{\calY}}{\nu} - 1\), which is the first variation (up to an additive constant) of the functional \(\rho \mapsto \frac{1}{2}\chi^{2}(\rho_{\calY} \| \nu)\).

\subsection{Connection between \textsf{SGA} and Maximum Mean Discrepancy}
The \ref{eqn:phi-match-update} abstraction also permits a similar interpretation of \ref{eqn:sga-update}, which corresponds to \(\Phi : f \mapsto e^{f}\) in \ref{eqn:phi-match-update}.
The map \(\calV_{\Phi}\) (\cref{eqn:proximal-point-V}) in this case is \(\pi_{\calY} - \nu\), which is the first variation of another notion of discrepancy between \(\pi_{\calY}\) and \(\nu\) given by a \emph{Maximum Mean Discrepancy} (abbrev. MMD).
Let \(k : \calZ \times \calZ \to \bbR\) be a positive-definite kernel, and let \(\calH_{k}\) be its RKHS.
We refer the reader to \citet[Chap. 4]{christmann2008support} for a more detailed exposition about kernels and their RKHS.
For \(\xi \in \calP(\calZ)\), the mean function w.r.t. kernel \(k\) is defined as
\begin{equation*}
    \frakm_{k}(\xi)(y) := \int_{\calY} k(y, y') \cdot \rmd\xi(y')~.
\end{equation*}
The MMD (for a kernel \(k\)) \citep{gretton2006kernel} is defined as
\begin{equation*}
    \mathrm{MMD}_{k}(\xi, \rho) = \sup_{\substack{f \in \calH_{k} \\ \|f\|_{\calH_{k}} \leq 1}} \left|\bbE_{\xi}[f] - \bbE_{\rho}[f]\right| = \|\frakm_{k}(\xi) - \frakm_{k}(\rho)\|_{\calH_{k}}~.
\end{equation*}
The first variation of \(\xi \mapsto \frac{1}{2}\mathrm{MMD}_{k}(\xi, \rho)^{2}\) is given by \(\frakm_{k}(\xi) - \frakm_{k}(\rho)\) \citep[Lem. 1]{mroueh19sobolev}.
Moreover, for \emph{characteristic} kernels \(k\) \citep{fukumizu2004dimensionality}, we have that \(\xi \mapsto \frakm_{\xi}\) is a one-to-one mapping, which establishes that \(\mathrm{MMD}_{k}\) for such kernels is a metric over the space of probability measures.
Examples of characteristic kernels are the identity, Gaussian, and Laplace kernels.

Going back to the map \(\calV_{\Phi}\) for \ref{eqn:sga-update}, this coincides with the first variation of 
\[\rho \mapsto \calL_{k}(\rho_{\calY}, \nu) := \frac{1}{2}\mathrm{MMD}_{k_{\mathrm{Id}}}(\rho_{\calY}, \nu)^{2}\] for the identity kernel \(k_{\mathrm{Id}}\) defined as \(k_{\mathrm{Id}}(y, y') = 1\) iff \(y = y'\), since \(\frakm_{k}(\xi)(y) = \xi(y)\).
This also motivates the consideration of a more general update \(\sfM^{k\textsf{-SGA}}\) for iteratively minimising \(\calL_{k}(\cdot; \nu)\) for any characteristic kernel, and results in the following update:
\begin{equation*}
    \sfM^{k\textsf{-SGA}}(\phi; \eta) := \phi + \eta \cdot \left\{\frakm_{k}(\nu) - \frakm_{k}(\pi(\phi, \phi^{+})_{\calY})\right\}~.\tag{\(k\)\textsf{-SGA}}\label{eqn:kernel-sga-update}
\end{equation*}

Due to the generality of the abstraction \ref{eqn:phi-match-update}, we can also view \ref{eqn:kernel-sga-update} as an instance of \ref{eqn:phi-match-update} with the choice \(\Phi_{k}(f) : f \mapsto e^{\frakm_{k}(f)}\) where we overload \(\frakm_k(f) := \int_{\calY} k(y, y') \cdot f(y')\rmd y'\).
As a consequence, \cref{corr:phi_rewrite} shows that the sequence of iterates \(\{\phi^{n}\}_{n \geq 1}\) formed by \ref{eqn:kernel-sga-update} results in a sequence of distributions \(\{\pi(\phi^{n}, (\phi^{n})^{+})\}_{n \geq 1}\) formed by iteratively applying \(\rootKL\) with \(\calF \leftarrow \calV_{\Phi_{k}}\). This implication is essential for deriving non-asymptotic rates of \ref{eqn:kernel-sga-update} in the next subsection.

\paragraph{\textsf{\(k\)-SGA} as kernelised \textsf{SGA}}
Since the map \(\frakm_{k}\) is additive, i.e., \(\frakm_{k}(f_{1} + f_{2}) = \frakm_{k}(f_{1}) + \frakm_{k}(f_{2})\) for suitable functions \(f_{1}\) and \(f_{2}\), we have
\begin{align*}
    \frakm_{k}(\delta J(\phi))(y) &= \int_{\calY} k(y, y') \cdot \delta J(\phi)(y') \rmd y' \\
    &= \int_{\calY} k(y, y') \cdot (\nu(y') - \pi(\phi, \phi^{+})_{\calY}(y')) \rmd y' \\
    &= \frakm_{k}(\nu)(y) - \frakm_{k}(\pi(\phi, \phi^{+})_{\calY})(y)~.
\end{align*}
Then, \ref{eqn:kernel-sga-update} can be equivalently written as
\begin{equation*}
    \sfM^{k\textsf{-SGA}}(\phi; \eta) = \phi + \eta \cdot \frakm_{k}(\delta J(\phi))
\end{equation*}
and can be viewed as performing kernel smoothing \(\delta J\) before using it for the update.

\subsubsection{Deriving non-asymptotic rates for \(k\)\textsf{-SGA}}
\label{sec:k_sga_rate}

With the new \(\calL_{k}\) objective, we derive non-asymptotic rates for \ref{eqn:kernel-sga-update} for a bounded, positive-definite kernel \(k\) here.
We state a key lemma from \cite{aubin2022mirror} which characterises the growth of \(\calL_{k}(\cdot; \nu)\) relative to the entropy functional \(H : \xi \mapsto \int_{\calY} \rmd\xi \log \rmd\xi\).
This, along with the convexity of \(\calL_{k}(\cdot; \nu)\), is instrumental in establishing a rate of convergence for \ref{eqn:kernel-sga-update}.

\begin{proposition}[{\cite[Prop. 14]{aubin2022mirror}}]
\label{prop:aubin-smooth}
Let \(k : \calY \times \calY \to \bbR\) be a bounded, positive definite kernel where \(c_{k} := \sup_{y \in \calY} k(y, y) < \infty\).
Then for any \(\xi, \overline{\xi} \in \calP(\calY)\),
\begin{equation*}
    0 \leq \left\langle \delta \calL_{k}(\overline{\xi}; \nu) - \delta \calL_{k}(\xi; \nu), \rmd\overline{\xi} - \rmd\xi\right\rangle \leq 2c_{k} \cdot \left\langle \delta H(\overline{\xi}) - \delta H(\xi), \rmd\overline{\xi} - \rmd\xi\right\rangle~.
\end{equation*}
Consequently,
\begin{equation*}
    0 \leq \calL_{k}(\overline{\xi}; \nu) - \calL_{k}(\xi; \nu) - \langle \delta \calL_{k}(\xi; \nu), \rmd\overline{\xi} - \rmd\xi\rangle \leq 2c_{k} \cdot \KLdist(\overline{\xi} \| \xi)~.
\end{equation*}
\end{proposition}
We now state the formal guarantee for \ref{eqn:kernel-sga-update}; the proof is given in \cref{sec:prf:kernel-sga-update-rate}.

\begin{theorem}
\label{thm:kernel-sga-update-rate}
Let \(\phi^{0} \in L^{1}(\nu)\), and let \(\{\phi^{n}\}_{n \geq 1}\) be the sequence of potentials obtained as \(\phi^{n + 1} = \sfM^{k\emph{\textsf{-SGA}}}\left(\phi^{n}; \min\left\{\frac{1}{2c_{k}}, 1\right\}\right)\) for a bounded, positive-definite kernel \(k\).
Define the sequence of distributions \(\{\pi^{n}\}_{n \geq 1}\) where \(\pi^{n} = \pi(\phi^{n}, (\phi^{n})^{+})\).
Then, for any \(N \geq 1\)
\begin{equation}
    \calL_{k}(\pi_{\calY}^{N}; \nu) \leq \frac{\max\{2c_{k}, 1\}}{N} \cdot \KLdist(\pi^{\star} \| \pi^{0})\, .
\end{equation}
\end{theorem}

Note that every \(\pi^{n}\) in the sequence of distributions generated by \ref{eqn:kernel-sga-update} stays in \(\calQ\).
When the kernel \(k\) is characteristic, \cref{thm:kernel-sga-update-rate} shows that \(\pi^{n}_{\calY}\) approaches \(\nu\) and consequently establishes a rate of convergence to the optimal coupling \(\pi^{\star}\).
Suppose \(\phi^{0} = \bm{0}\), and \(\pi^{0} = \pi(\phi^{0}, (\phi^{0})^{+})\).
Then from \citet[Proof of Cor. 1]{leger2021gradient}, we know that \(\KLdist(\pi^{\star} \| \pi^{0}) \leq \KLdist(\pi^{\star} \| \pi^{\textsf{ref}})\) where \(\pi^{\textsf{ref}}\) is the reference measure in \cref{eqn:static_SB}.
Consequently, after \(N\) iterations, the rate we obtain from \cref{thm:kernel-sga-update-rate} is
\begin{equation*}
    \frac{2c_{k}}{N} \cdot \frac{\KLdist(\pi^{\star} \| \pi^{\textsf{ref}})}{\varepsilon}~.
\end{equation*}
This highlights the better dependence on \(\varepsilon\) compared to the more classical analyses of the Sinkhorn algorithm where the dependence on \(\varepsilon\) is of the form \(e^{-\varepsilon^{-1}}\).
From this, we can directly infer that \ref{eqn:sga-update} results in a sequence of joint distributions whose \(\calY\)-marginals converge to \(\nu\) in the squared MMD defined by a bounded, positive definite kernel at a \(\frac{1}{N}\) rate \emph{without} additional assumptions on either the marginal \(\nu\) or the cost function \(c(\cdot, \cdot)\), since the bound is meaningful as soon as an optimal coupling exists $\KLdist(\pi^{\star} \| \pi^{\textsf{ref}})<\infty$.

\subsection{Additional remarks on \(\Phi\)\textsf{-match}}

\paragraph{\ref{eqn:phi-match-update} as a mirror method}

The interpretation of \ref{eqn:phi-match-update} presented here is a generalisation of the mirror method perspective of \(\eta\)-Sinkhorn introduced in \cite{karimi2024sinkhorn}.
Let \(\varphi : \calP(\calX \times \calY) \to \bbR\) be given by \(\varphi(\pi) = \KLdist(\pi \| \pi^{\textsf{ref}})\), where \(\pi^{\textsf{ref}}\) is the reference measure in the static Schr\"{o}dinger bridge problem (\cref{eqn:static_SB}).
Then,
\begin{equation*}
    \delta \varphi(\pi)(x, y) = \log \frac{\pi(x, y)}{\pi^{\textsf{ref}}(x, y)}~.
\end{equation*}
Following \citet{karimi2024sinkhorn}, consider the restriction of \(\varphi\) to \(\overline{\calQ} = \{\pi \in \calP(\calX \times \calY) : \pi_{\calX} = \mu\}\), and let \(\varphi^{\star}\) be the convex conjugate of this restricted \(\varphi\).
From \citet[Lem. 3]{karimi2024sinkhorn}, we have that for any suitably integrable \(h\),
\begin{equation}
\label{eqn:special_MD_map}
    \delta\varphi^{\star}(h)(x, y) = \mu(x) \cdot \frac{\pi^{\textsf{ref}}(x, y) \exp(h(x,y))}{\int_{\calY} \pi^{\textsf{ref}}(x, y') \exp(h(x, y')) \rmd y'}~.
\end{equation}
Note that \(\delta \varphi^{\star}(h)\) is a valid density function over \(\calX \times \calY\) since \(\int_{\calY} \delta \varphi^{\star}(h)(x, y)\rmd y = \mu(x)\).
Additionally, as expected of a mirror map, we have \(\delta \varphi^{\star}(\delta \varphi(\pi)) = \pi\) for any \(\pi \in \overline{\calQ}\).

\begin{lemma}
\label{lem:md_interpret}
Let \(\phi^{0} \in L^{1}(\nu)\) and let \(\{\phi^{n}\}_{n \geq 1}\) be the sequence of potentials obtained as \(\phi^{n + 1} = \sfM^{\Phi\emph{\textsf{-match}}}(\phi^{n}; \eta)\).
Then, the sequence of distributions \(\{\pi^{n}\}_{n \geq 0}\) where \(\pi^{n} = \pi(\phi^{n}, (\phi^{n})^{+})\) satisfy for every \(n \geq 0\)
\begin{equation}
\label{eqn:discrete_MD}
    \pi^{n + 1} = \delta \varphi^{\star}\left(\delta \varphi(\pi^{n}) - \eta \cdot \calV_{\Phi}(\pi^{n})\right)
\end{equation}
where \(\calV_{\Phi}\) is defined in \cref{eqn:proximal-point-V}.
\end{lemma}

\cref{lem:md_interpret} offers a mirror descent interpretation of \ref{eqn:phi-match-update} in the spirit of a mirror descent method in finite-dimensional optimisation \citep{nemirovskii1983problem}, where one uses mirror map to perform updates in the dual space.
We give the proof of \cref{lem:md_interpret} in \cref{sec:prf:md_interpret}.

\section{Extensions}
\label{sec:extensions}
Here, we build on the ideas in the preceding sections and present two extensions.
The first extension revisits the smoothness property of the semi-dual discussed previously and illustrates how this can be leveraged to design iterative methods for maximising the semi-dual with provable guarantees, which also allows us to design an accelerated version of one of these algorithms without having to place assumptions on \(\mu\) and \(\nu\).
The second extension adapts \ref{eqn:phi-match-update} for the dynamical Schr\"{o}dinger bridge problem which generalises the static Schr\"{o}dinger bridge problem to path measures.

\subsection{Adapting to smoothness of the semi-dual \(J\)}
\label{sec:extensions:steep-ascent}

\paragraph{Signed semi-dual gradient ascent}
As mentioned previously, the fundamental obstacle to establishing guarantees for \ref{eqn:sga-update} is the mismatch between the inner product in the Bregman divergence (which is in \(L^{2}(\calY)\)) and the squared growth (which is in \(L^{\infty}(\calY)\)).
This suggests a better-suited steepest ascent update involving \(L^{1}(\calY)\) (which corresponds to the dual norm of \(L^{\infty}(\calY)\)) that is defined by
\begin{equation*}
    \phi^{n + 1} = \sfM^{\textsf{sign-SGA}}(\phi^{n}; \eta) := \phi^{n} + \eta \cdot \|\delta J(\phi^{n})\|_{L^{1}(\calY)} \cdot \mathrm{sign}(\delta J(\phi^{n}))~.\tag{\textsf{sign-SGA}}\label{eqn:sign-sga-update}
\end{equation*}
Note that for any \(\phi \in L^{1}(\nu) \cap L^{\infty}(\calY)\), \(\sfM^{\textsf{sign-SGA}}(\phi; \eta) \in L^{1}(\nu) \cap L^{\infty}(\calY)\) due to the form of \(\delta J(\phi)\) in \cref{eqn:first-var-J}.
This is because for any \(\phi \in L^{1}(\nu)\), \(\|\delta J(\phi)\|_{L^{1}(\calY)} = \|\pi(\phi, \phi^{+})_{\calY} - \nu\|_{L^{1}(\calY)} \leq 2\), and the sign function being bounded pointwise.
Also, for a sufficiently small \(\eta > 0\), the growth property implied by \cref{lem:Linf-smooth} asserts that \(J(\sfM^{\textsf{sign-SGA}}(\phi; \eta)) \geq J(\phi)\) since one can show
\begin{equation*}
    J(\sfM^{\textsf{sign-SGA}}(\phi; \eta)) \geq J(\phi) + \eta \cdot \|\delta J(\phi)\|_{L^{1}(\calY)}^{2} - \frac{\eta^{2}}{2} \cdot \|\delta J(\phi)\|_{L^{1}(\calY)}^{2}~.
\end{equation*}
This ascent property in conjunction with the concavity of \(J\) enables us to give a non-asymptotic convergence rate for \ref{eqn:sign-sga-update} with an ``anchoring'' step.
Let \(y_{\texttt{anc}} \in \calY\) be an anchor point, and \(\phi^{0} \in L^{1}(\nu) \cap L^{\infty}(\calY)\) be such that \(\phi^{0}(y_{\texttt{anc}}) = 0\).
Define the set
\begin{equation*}
    \calT_{\phi^{0}, y_{\texttt{anc}}} := \{\phi \in L^{1}(\nu) \cap L^{\infty}(\calY) : \phi(y_{\texttt{anc}}) = 0, J(\phi) \geq J(\phi^{0})\}~.
\end{equation*}

\begin{theorem}
\label{thm:sign-sga-conv-rate}
Let \(\{\phi^{n}\}_{n \geq 1}\) be the sequence of potentials generated according to the following recursion for \(n \geq 0\):
\begin{align*}
\phi^{n + \nicefrac{1}{2}} &= \sfM^{\emph{\textsf{sign-SGA}}}(\phi^{n}; \eta) \\
\phi^{n + 1} &= \phi^{n + \nicefrac{1}{2}} - (\phi^{n + \nicefrac{1}{2}}(y_{\emph{\texttt{anc}}}) - \phi^{n}(y_{\emph{\texttt{anc}}})) \cdot \bm{1}
\end{align*}
with \(\phi^{0}, y_{\emph{\texttt{anc}}}\) as defined above.
Then, for all \(N \geq 1\), \(\phi^{N} \in \calT_{\phi^{0}, y_{\emph{\texttt{anc}}}}\) and for \(\widetilde{\phi}^{\star} = \argmax~\{J(\phi) : \phi \in  \calT_{\phi^{0}, y_{\emph{\texttt{anc}}}}\}\) we have
\begin{equation*}
    J(\phi^{N}) - J(\widetilde{\phi}^{\star}) \geq - \frac{2 \cdot \mathrm{diam}(\calT_{\phi^{0}, y_{\texttt{anc}}}; L^{\infty}(\calY))^{2}}{N + 1}~,\quad \text{where }~~ \mathrm{diam}( \calS; L^{\infty}(\calY)) := \sup_{\overline{\phi}, \phi \in \calS} \|\overline{\phi} - \phi\|_{L^{\infty}(\calY)}~.
\end{equation*}
\end{theorem}

The first step applies \ref{eqn:sign-sga-update}, and the second step recenters the iterate to satisfy \(\phi^{n + 1}(y_{\texttt{anc}}) = \phi^{n}(y_{\texttt{anc}})\).
The recentering does not affect the value of the semi-dual as noted previously in \cref{rmk:J_shift_invariance}, and if \(\phi^{\star} \in L^{\infty}(\calY)\) is a Schr\"{o}dinger potential, then \(\phi^{\star} - \phi^{\star}(y_{\texttt{anc}}) \cdot \bm{1}\) is also a maximiser of \(J\), which lies in \( \calT_{\phi^{0}, y_{\texttt{anc}}}\).
This shift-invariance also explains the use of the anchoring step: without anchoring, the superlevel set of \(J\) is unbounded.
We can hence infer that the sequence \(\{J(\phi^{n})\}_{n\geq 1}\) converges to the maximum of the semi-dual \(J\). The proof of \cref{thm:sign-sga-conv-rate} can be found in \cref{app:sec:prf:sign-sga-conv-rate}.

\paragraph{Projected semi-dual gradient ascent}
When the cost function is bounded in a certain manner, it is possible to show that the semi-dual satisfies a different notion of smoothness, but non-uniformly depending on the ``size'' of the domain considered. 
\begin{lemma}
\label{lem:L2nu-smooth}
Let \(\phi, \overline{\phi} \in L^{2}(\nu)\) be such that \(\|\phi\|_{L^{\infty}(\calY)}, \|\overline{\phi}\|_{L^{\infty}(\calY)} \leq B\) for a given \(B > 0\).
Assume that the cost \(c(\cdot, \cdot) \geq 0\).
Then,
\begin{equation*}
    J(\overline{\phi}) - J(\phi) - \langle \delta J(\phi), \overline{\phi} - \phi\rangle \geq -\frac{\lambda(B) \cdot \|\overline{\phi} - \phi\|_{L^{2}(\nu)}^{2}}{2}~; ~~\lambda(B) = e^{2B} \cdot \bbE_{(x, y') \sim \mu \otimes \nu}\left[\exp\left(\frac{c(x, y')}{\varepsilon}\right)\right]~.
\end{equation*}
\end{lemma}

While \cref{lem:Linf-smooth} is a general statement, regularity condition \cref{lem:L2nu-smooth} is parameterised by a ``size'' parameter \(B\), and hence it is useful to understand what a reasonable choice of \(B\) is for the purposes of solving the eOT problem.
If \(B\) is too small, then it is likely that the Schr\"{o}dinger potential \(\phi^{\star}\) would not satisfy \(\|\phi^{\star}\|_{L^{\infty}(\calY)} \leq B\).
Interestingly however, the Schr\"{o}dinger potentials \(\phi^{\star}\) and \(\psi^{\star} = (\phi^{\star})^{+}\) inherit properties from the cost function \(c(\cdot, \cdot)\), which allow us to determine a reasonable choice of \(B\) based on the cost function.
This is formalised in the following proposition.
\begin{proposition}[{\cite[Lem. 2.7]{marino2020optimal}}]
\label{prop:linf-bound-poten}
Consider the dual eOT problem defined in \cref{eqn:eOT-prob-dual}.
There exists Schr\"{o}dinger potentials \(\phi^{\star}, \psi^{\star}\) such that
\begin{equation*}
    \|\phi^{\star}\|_{L^{\infty}(\calY)} \leq \frac{3\|c\|_{L^{\infty}(\calX \times \calY)}}{2}; \quad \|\psi^{\star}\|_{L^{\infty}(\calX)} \leq \frac{3\|c\|_{L^{\infty}(\calX \times \calY)}}{2}~.
\end{equation*}
\end{proposition}
The intriguing aspect of this proposition is the lack of a dependence on the regularisation parameter \(\varepsilon > 0\).
This proposition also suggests that solving the semi-dual problem for eOT over the space of functions \(\phi \in L^{2}(\nu)\) such that \(\|\phi\|_{L^{\infty}(\calY)} \leq \frac{3\|c\|_{L^{\infty}(\calX \times \calY)}}{2}\) is sufficient to recover a Schr\"{o}dinger potential.

Let \(\calS_{B} := \{\phi \in \calL^{2}(\nu) : \|\phi\|_{L^{\infty}(\calY)} \leq B\} \subset L^{1}(\nu)\).
We define the following operation
\begin{equation*}
    \phi^{n + 1} = \sfM^{\textsf{proj-SGA}}_{\calS_{B}}(\phi; \eta) := \argmin_{\overline{\phi} \in \calS_{B}}~ \left\|\overline{\phi} - \left(\phi + \eta \cdot \frac{\delta J(\phi)}{\nu}\right)\right\|_{L^{2}(\nu)}^{2}~.\tag{\textsf{proj-SGA}}\label{eqn:proj-sga-update}
\end{equation*}
While this operation may appear fortuitous, this is actually a natural recommendation based on \cref{lem:L2nu-smooth}.
This is because it can be obtained as the solution to a truncated local quadratic approximation of the semi-dual given below (truncated due to the restriction to \(\calS_{B}\)), which is inspired by ISTA \citep{beck2009fast}:
\begin{equation*}
    \sfM^{\mathsf{proj-SGA}}_{\calS_{B}}(\phi; \eta) = \argmax_{\overline{\phi} \in \calS_{B}}~J(\phi) + \langle \delta J(\phi), \overline{\phi} - \phi\rangle - \frac{1}{2\eta} \cdot \|\overline{\phi} - \phi\|_{L^{2}(\nu)}^{2}~.
\end{equation*}
From \cref{lem:L2nu-smooth}, when the cost \(c(\cdot, \cdot)\) is non-negative and the step size satisfies \(\eta \leq \lambda(B)^{-1}\), \(J(\sfM^{\textsf{proj-SGA}}_{\calS_{B}}(\phi; \eta)) \geq J(\phi)\) for any \(\phi \in \calS_{B}\) as
\begin{align*}
    J(\sfM^{\textsf{proj-SGA}}_{\calS_{B}}(\phi; \eta)) &\geq J(\phi) + \left\langle \delta J(\phi), \sfM^{\textsf{proj-SGA}}_{\calS_{B}}(\phi; \eta) - \phi\right\rangle - \frac{\lambda(B)}{2}\|\sfM^{\textsf{proj-SGA}}_{\calS_{B}}(\phi; \eta) - \phi\|_{L^{2}(\nu)}^{2} \\
    &\geq J(\phi) + \left\langle \delta J(\phi), \sfM^{\textsf{proj-SGA}}_{\calS_{B}}(\phi; \eta) - \phi\right\rangle - \frac{1}{2\eta}\|\sfM^{\textsf{proj-SGA}}_{\calS_{B}}(\phi; \eta) - \phi\|_{L^{2}(\nu)}^{2} \\
    &\geq J(\phi)~.
\end{align*}
The final step uses the optimality of \(\sfM^{\textsf{proj-SGA}}_{\calS_{B}}(\phi; \eta)\).
Analogous to \ref{eqn:sign-sga-update}, the concavity of \(J\) results in the following non-asymptotic convergence guarantee for \ref{eqn:proj-sga-update}.

\begin{theorem}
\label{thm:proj-sga-conv-rate}
Suppose \(c(\cdot, \cdot)\) is a non-negative cost function such that \(\lambda(B) < \infty\).
Let \(\{\phi^{n}\}_{n \geq 1}\) be the sequence of potentials generated according to the following recursion for \(n \geq 0\):
\begin{equation*}
    \phi^{n + 1} = \sfM^{\emph{\textsf{proj-SGA}}}_{\calS_{B}}(\phi^{n}; \lambda(B)^{-1})~,
\end{equation*}
where \(\lambda(B)\) is defined in \cref{lem:L2nu-smooth}.
Then, for all \(N \geq 1\), \(\phi^{N} \in \calS_{B}\) and for \(\widetilde{\phi}^{\star} = \argmax\limits_{\phi \in \calS_{B}} J(\phi)\)
\begin{equation*}
    J(\phi^{N}) - J(\widetilde{\phi}^{\star}) \geq -\frac{\lambda(B) \cdot \|\phi^{0} - \widetilde{\phi}^{\star}\|_{L^{2}(\nu)}^{2}}{2N}~\, .
\end{equation*}
\end{theorem}

\paragraph{Accelerated \textsf{proj-SGA}}
Note that the growth condition in \cref{lem:L2nu-smooth} that \ref{eqn:proj-sga-update} is conceptually based on is given in terms of \(L^{2}(\nu)\) which is a Hilbert space.
This encourages adopting the structure of FISTA \citep{beck2009fast} to design an accelerated version of \ref{eqn:proj-sga-update}, which would lead to a rate of convergence that scales as \(\nicefrac{1}{N^{2}}\).
This method is based on minor adjustments to handle the non-uniform growth condition in \cref{lem:L2nu-smooth}, and is presented below.
For initial values \(\phi^{1} = \overline{\phi}^{0} \in \calS_{B}\), and \(t_{1} = 1\):
\begin{align*}
    &\overline{\phi}^{n} = \sfM^{\textsf{proj-SGA}}_{\calS_{B}}(\phi^{n}; \lambda(3B)^{-1})~; ~~ t_{n + 1} = \frac{1 + \sqrt{1 + 4t_{n}^{2}}}{2}~; \\
    &\phi^{n + 1} = \overline{\phi}^{n} + \left(\frac{t_{n} - 1}{t_{n + 1}}\right) \cdot (\overline{\phi}^{n} - \overline{\phi}^{n - 1})~.
    \tag{\textsf{proj-SGA++}}\label{eqn:acc-proj-sga-update}
\end{align*}

\begin{theorem}
\label{thm:acc-proj-sga-conv-rate}
Suppose \(c(\cdot, \cdot)\) is a non-negative cost function such that \(\lambda(B) < \infty\).
Consider the sequences \(\{\phi^{n}\}_{n \geq 2}, \{\overline{\phi}^{n}\}_{n \geq 1}\) generated according to \ref{eqn:acc-proj-sga-update}.
Then, for any \(N \geq 1\),
\begin{equation*}
    J(\overline{\phi}^{N}) - J(\widetilde{\phi}^{\star}) \geq -\frac{2 \cdot \lambda(3B) \cdot \|\overline{\phi}^{0} - \tilde{\phi}^{\star}\|_{L^{2}(\nu)}^{2}}{(N + 1)^{2}}; \quad \widetilde{\phi}^{\star} \in \argmax_{\phi \in \calS_{B}}\, J(\phi)~.
\end{equation*}
\end{theorem}

We give the proof of \cref{thm:proj-sga-conv-rate,thm:acc-proj-sga-conv-rate} in \cref{app:sec:prf:proj-sga-conv-rate,app:sec:prf:acc-proj-sga-conv-rate}.
From \cref{prop:linf-bound-poten}, we know that the maximum value of the semi-dual \(J\) over \(\calS_{B}\) for \(B = \frac{3\|c\|_{L^{\infty}(\calX \times \calY)}}{2}\) is \(J(\phi^{\star})\) for the Schr\"{o}dinger potential \(\phi^{\star}\).
This implies that the sequences of semi-dual values generated by \ref{eqn:proj-sga-update} and \ref{eqn:acc-proj-sga-update} with the appropriate step sizes converge to \(J(\phi^{\star})\) at a \(\frac{1}{N}\) and \(\frac{1}{N^{2}}\) rate respectively.
A crude bound on \(\lambda(B)\) in this setting is given by \(\lambda(B) \leq \exp\left(B \cdot (\varepsilon^{-1} + 2)\right)\).

We would like to mention that this is not the only accelerated method for the eOT problem.
One can take advantage of the structure of \(\calX\) and \(\calY\), particularly when they are discrete spaces to directly accelerate \(\sfM^{\textsf{SGA}}\) instead of \(\sfM^{\textsf{proj-SGA}}_{\calS_{B}}\) analogous to accelerating gradient descent to minimise a convex, smooth function in finite dimension.
In that special case, the semi-dual is concave and also smooth in the canonical sense as implied by \cref{lem:Linf-smooth} due to the monotonicity of norms, and this leads to a rate that scales as \(\frac{1}{N^{2}}\).
Alternatively, one could possibly also design accelerated algorithms for minimising \(\rho \mapsto \calL_{k}(\rho_{\calY}; \nu)\) subject to the constraint \(\rho \in \calQ\) in contrast to \ref{eqn:acc-proj-sga-update} which is based on the semi-dual.
The constraint ensures that the solution has the form of the optimal coupling \(\pi^{\star}\). However, designing an accelerated method for this problem in the flavour of accelerated MD for constrained optimisation can prove challenging, due to the non-convexity of the set \(\calQ\) and the need for linear combination of past iterates when considering momentum.
In \cref{sec:md-sga-acceleration}, we elaborate on the challenge of building acceleration based on this primal design.

We conclude this discussion by drawing a comparison to \emph{how} the rates for \ref{eqn:sga-update}, \ref{eqn:sign-sga-update}, and \ref{eqn:proj-sga-update} have been established.
Specifically, the techniques used to prove \cref{thm:sign-sga-conv-rate,thm:proj-sga-conv-rate,thm:acc-proj-sga-conv-rate} are fundamentally different from \cref{thm:kernel-sga-update-rate}.
This highlights the benefits of considering alternate viewpoints for establishing provable guarantees.

%demonstrate that even the continuous-time analogue of \ref{eqn:sga-update} doesn't ensure that solutions belong to \(\calQ\).

\subsection{Path-space \textsf{\(\Phi\)-match} for the Schr\"{o}dinger Bridge problem}
\label{sec:extensions:SB}

The eOT problem in the static Schr\"{o}dinger bridge form (\cref{eqn:static_SB}) can be viewed as finding a certain distribution in \(\calP(\calX \times \calY)\) closest to another reference distribution in \(\calP(\calX \times \calY)\).
This can also be posed for distributions over curves \(\calC([0, T]; \calZ)\) for \(\calZ \subseteq \bbR^{d}\) more generally, and a path measure precisely captures the notion of a probability measure over trajectories.
For a stochastic process \((X_{t})_{t \in [0, T]}\) with state space \(\calZ \subseteq \bbR^{d}\), the path measure is a collection of distributions \((\bbP_{t})_{t \in [0, T]}\) where \(\bbP_{t}\) is the law of \(X_{t}\).
The dynamical Schr\"{o}dinger bridge problem is formulated to find the path measure with endpoints \(\bbP_{0} = \mu\) and \(\bbP_{T} = \nu\) that is closest to the reference path measure.
Formally,
\begin{equation}
\label{eqn:SB}
\min_{\bbP} ~\KLdist(\bbP \| \bbP^{\textsf{ref}}) \quad \text{such that } \bbP_{0} = \mu, \bbP_{T} = \nu~.
\end{equation}
This is a (strictly) convex problem defined in  \(\calP(\calC([0, T], \calZ))\).
The classical result of \citet{follmer1988random} shows that by the disintegration property of the KL divergence, the optimal path measure \(\bbP^{\star}\) that solves \cref{eqn:SB} can be decoupled as
\begin{equation*}
    \bbP^{\star}_{\{0, T\}} = \pi^{\star}~; \quad \bbP^{\star}_{(0, T) | X_{0} , X_{T} } = \bbP^{\textsf{ref}}_{(0, T) | X_{0}, X_{T} }%~ \forall x, y \in \calZ
\end{equation*}
where \(\pi^{\star}\) is the solution to the (primal) eOT problem in the static Schr\"{o}dinger bridge form (\cref{eqn:static_SB}) with \(\pi^{\textsf{ref}} = \bbP^{\textsf{ref}}_{0, T}\), and \(\bbP_{(0, T) | X_{0}, X_{T} }\) is the terminal-conditioned measure (also termed the bridge). % such that \(X_{0} = x, X_{T} = y\) 
Due to the above disintegration property and the form of the solution of the static Schr\"{o}dinger bridge problem in \cref{eqn:eOT-opt-form}, under certain regularity conditions (see \citet[Thms. 2.8 and 2.9]{leonard2014survey}), there exists measurable functions \(\bm{\psi}_{0}\) and \(\bm{\phi}_{T}\) such that the relative density of \(\bbP^{\star}\) w.r.t. \(\bbP^{\textsf{ref}}\) satisfies
\begin{equation}
\label{eqn:path-factorise}
    \frac{\rmd\bbP^{\star}_{[0, T]}}{\rmd\bbP^{\textsf{ref}}_{[0, T]}}((X_{t})_{t \in [0, T]}) = \exp\left(\bm{\phi}_{T}(X_{T}) - \bm{\psi}_{0}(X_{0})\right)~.
\end{equation}
Moreover, if \(\bbP^{\textsf{ref}}\) is Markovian i.e., if the underlying stochastic process is a Markov process, then the factorisation in \cref{eqn:path-factorise} is necessary and sufficient for the characterisation of \(\bbP^{\star}\) and is unique.
This parallels the unique form of the optimal coupling form in \cref{eqn:eOT-opt-form}.

An example of a reversible Markov process \(\bbP^{\textsf{ref}}\) is the reversible Brownian motion on \([0, T]\).
In this case, the density of \(\bbP^{\textsf{ref}}_{\{0, T\}}\) satisfies \(\bbP^{\textsf{ref}}_{\{0, T\}}(x, y) \propto \exp\left(-\frac{\|x - y\|^{2}}{2T}\right)\).
Finding \(\pi^{\star}\) is equivalent to solving the eOT problem (\cref{eqn:eOT-prob-primal}) with cost \(c(x, y) = \frac{\|x - y\|^{2}}{2}\) and \(\varepsilon = T\), and the bridge in this setting corresponds to the classical Brownian bridge which draws a more substantive connection between eOT and the dynamic SB problem.
However, in general, it is not feasible to directly sample from the bridge, and this motivates the design of generic methods for solving \cref{eqn:SB}.
\paragraph{Iterative Proportional Fitting}
A popular algorithm for solving the dynamic Schr\"{o}dinger bridge problem is commonly referred to as the \emph{Iterative Proportional Fitting} (IPF) algorithm defined by the iteration below.
\begin{equation*}
    \bbP^{n+\nicefrac{1}{2}} = \argmin_{\overline{\bbP}}\left\{ \KLdist(\overline{\bbP} \| \bbP^{n}) ~:~ \overline{\bbP}_{T} = \nu\right\}~, \quad \mathbb{P}^{n+1} = \argmin_{\overline{\bbP}} \left\{ \KLdist(\overline{\bbP} \| \bbP^{n+\nicefrac{1}{2}}) ~:~ \overline{\bbP}_{0} = \mu\right\}~.
\end{equation*}
This can be viewed as the path-space analogue of the alternating projection form of the Sinkhorn algorithm.
\citet{ruschendorf1995convergence} shows that the initial and final marginals of the sequence of iterates \(\{\bbP^{n}\}_{n \geq 1}\) generated by the above iteration with \(\bbP^{0} = \bbP^{\textsf{ref}}\) converges to \(\mu\) and \(\nu\) respectively as \(n \to \infty\), and \citet{bernton2019schr} gives a non-asymptotic rate of convergence that scales as \(\frac{1}{N}\) in the number of iterations \(N\) building off the result by \citet{ruschendorf1995convergence}.
In \citet[Prop. 2]{karimi2024sinkhorn}, it is shown that this iteration is the solution to the following local greedy update:
\begin{equation*}
    \bbP^{n + 1} = \argmin \left\{\langle \delta \KLdist(\bbP^{n}_{T} \| \nu), \overline{\bbP} - \bbP^{n}\rangle_{L^{2}(\calC([0, T], \calZ))} + \KLdist(\overline{\bbP} \| \bbP^{n}) ~:~ \overline{\bbP}_{0} = \mu\right\}~.
\end{equation*}

\paragraph{From IPF to \textsf{path-\(\Phi\)-match}}
The similarities between the local greedy update above and \cref{eqn:proximal-point} motivate us to propose a path-space analogue of \ref{eqn:phi-match-update}.
For \(\eta \in [0, 1]\), define
\begin{equation}
\label{eqn:proximal-point-path-space}
    \bbP^{n + 1} = \argmin_{\overline{\bbP}}~\left\{\langle \overline{\calV}_{\Phi}(\bbP^{n}),~ \overline{\bbP} - \bbP^{n}\rangle + \eta^{-1} \cdot \KLdist(\overline{\bbP} \| \bbP^{n}) ~:~ \overline{\bbP}_{0} = \mu\right\}~
\end{equation}
where \(\overline{\calV}_{\Phi}(\bbP) = \log \Phi(\bbP_{T}) - \log \Phi(\nu)\).
When \(\Phi(f) = f\), \(\overline{\calV}_{\Phi}\) is the first variation of \(\rho \mapsto \KLdist(\rho_{T} \| \nu)\), and \cref{eqn:proximal-point-path-space} recovers the interpretation of IPF stated above with $\eta=1$.
Similar to the equivalence in \cref{lem:proximal-point}, it can be shown that this local greedy update can be expressed a sequence of alternating projections but in the space of path measures as written below.
\begin{subequations}
\makeatletter
\def\@currentlabel{\textsf{path-\(\Phi\)-match}}
\makeatother
\label{eqn:path-phi-match-update}
\renewcommand{\theequation}{\textsf{path-\(\Phi\)-match}(\alph{equation})}
\begin{align}
    \bbP^{n + \nicefrac{1}{2}} &= \argmin_{\overline{\bbP}} \left\{\KLdist(\overline{\bbP} \| \bbP^{n}) ~:~ \overline{\bbP}_{T} \propto \bbP^{n}_{T} \cdot \frac{\Phi(\nu)}{\Phi(\bbP^{n}_{T})}\right\} \label{eqn:first_half_IPFP}\\
    \bbP^{n + 1} &= \argmin_{\overline{\bbP}}  \left\{\eta \cdot \KLdist(\overline{\bbP} \| \bbP^{n + \nicefrac{1}{2}}) + (1 - \eta) \cdot \KLdist(\overline{\bbP} \| \bbP^{n})~:~ \overline{\bbP}_{0} = \mu\right\}~.\label{eqn:second_half_IPFP}
\end{align}
\end{subequations}
As it turns out, \ref{eqn:path-phi-match-update} (and therefore path measure update \cref{eqn:proximal-point-path-space}) can be implemented as updates on the drifts of a sequence of SDEs whose corresponding solutions are given by \(\{\bbP^{n}\}_{n \geq 1}\). Let \(\bbP^{\textsf{ref}}\) be the strong solution of an SDE (hence Markovian) as 
\begin{equation}
\label{eqn:ref_process_sde}
    \rmd X_t = u^{\textsf{ref}}(X_t)\rmd t + \sqrt{2} \cdot \rmd B_t~, \quad X_{0} \sim \mu~.
\end{equation}
By Girsanov's theorem \cite[Chap. 3]{pavliotis2014stochastic}, the dynamical SB problem can be reduced to a minimisation problem over path measures \(\bbP\) induced by SDEs with a different drift vector field \((v_{t})_{t \geq 0}\)
\begin{equation}
\label{eqn:posterior_process_sde}
    \rmd X_{t} = v_{t}(X_{t})\rmd t + \sqrt{2}\cdot \rmd B_{t}~, \quad X_0\sim \mu~.
\end{equation}
We formally characterise this recursive update on $\{v_t^n\}_{n\geq 1, t\in [0,T]}$ in the following proposition.

\begin{proposition}
\label{prop:path-space-phi-match-update-SDE}
Let \(\{\bbP^{n}\}_{n \geq 1}\) be obtained from \emph{\ref{eqn:path-phi-match-update}} with \(\bbP^{0} = \bbP^{\emph{\textsf{ref}}}\).
Then for every \(n \geq 1\), \(\bbP^{n}\) is the path measure associated with the solution of the following SDE:
\begin{equation*}
    \rmd X_{t} = v_{t}^{n}(X_{t})\rmd t + \sqrt{2} \cdot \rmd B_{t}~,\quad X_{0} \sim \mu~,
\end{equation*}
where \(v_{t}^{0} = u^{\emph{\textsf{ref}}}\) and for every \(n \geq 0\),
\begin{equation*}
    v_{t}^{n + 1} = v_{t}^{n} + 2\eta \cdot (\nabla \log p_{t}^{n + \nicefrac{1}{2}} - \nabla \log p_{t}^{n}) - 2 \cdot \nabla V_{t}
\end{equation*}
with \(p^{n + \nicefrac{1}{2}}_{t}\) and \(p^{n}_{t}\) denoting marginal densities of \(\bbP^{n + \nicefrac{1}{2}}\) and \(\bbP^{n}\) respectively. Above \(V_{t}\) is defined as
\begin{equation*}
    V_{t}(x) := -\log \bbE \left[\left.\exp\left(-\eta \cdot (1 - \eta) \cdot \int_{t}^{T} \|\nabla \log p_{s}^{n + \nicefrac{1}{2}}(Z_{s}) - \nabla \log p_{s}^{n}(Z_{s})\|^{2} \rmd s\right) ~\right|~ Z_{t} = x\right]
\end{equation*}
where the expectation is taken over the law of the SDE $(Z_s)_{s\geq t}$ given by
\begin{equation*}
    \rmd Z_{s} = [v_{s}^{n}(Z_{s}) + 2\eta \cdot (\nabla \log p_{s}^{n + \nicefrac{1}{2}}(Z_{s}) - \nabla \log p_{s}^{n}(Z_{s}))]~\rmd s + \sqrt{2} \cdot \rmd B_{s}~;\; Z_{t} = x~.
\end{equation*}
\end{proposition}

The above proposition extends \citet[Thm. 4.2]{karimi2024sinkhorn} for IPF to \ref{eqn:path-phi-match-update}.
The key difference here is that the path measure \(\bbP^{n + \nicefrac{1}{2}}\) is defined by the time-reversal of the SDE with respect to \(\bbP^{n}\), but with initial distribution whose density is proportional to \(\bbP^{n}_{T} \cdot \frac{\Phi(\nu)}{\Phi(\bbP_{T}^{n})}\), instead of \(\nu\) as for the case of $\eta$-IPF.
\citet[Sec. C.2]{karimi2024sinkhorn} discuss approaches for computing \(\nabla V_{t}\) through stochastic optimal control, and the other terms $\nabla \log p_t^{n+1/2},\nabla \log p_t^n$ can be obtained through score-matching time-reversals as in the classical IPF algorithm where $\eta=1$. We give details of the proof of this proposition in \cref{app:sec:prf:SB-details}.

The connection to \ref{eqn:phi-match-update} can also be seen in the following manner.
Suppose \(\bbP^{n}\) satisfies the factorisation form in \cref{eqn:path-factorise} with some \(\bm{\phi}^{n}_{T}\) and \(\bm{\psi}^{n}_{0}\), and note that \(\bbP^{n}_{0} = \mu\).
Then, one can show analogous to \cref{lem:KL_primal_projection} that \(\bbP^{n + 1}\) admits the factorisation form in \cref{eqn:path-factorise} where
\begin{equation*}
    \bm{\phi}^{n + 1}_{T} = \bm{\phi}^{n}_{T} - \eta \cdot (\log \Phi(\bbP^{n}_{T}) - \log \Phi(\nu))~
\end{equation*}
and \(\bm{\psi}_{0}^{n + 1}\) set to satisfy \(\bbP^{n + 1}_{0} = \mu\) is equivalent to the output of \ref{eqn:path-phi-match-update}.
From this, it can be seen that for each \(n \geq 0\), \(\bbP^{n + 1}_{\{0, T\}}\) is the solution for the static Schr\"{o}dinger bridge problem with marginals \((\mu, \bbP^{n}_{T})\), and \(\bbP^{n + 1}_{(0, T) | X_{0}, X_{T}} = \bbP^{\textsf{ref}}_{(0, T) | X_{0}, X_{T}}\).
As a consequence, if \(\bbP^{n}_{T} \to \nu\) then \(\bbP^{n} \to \bbP^{\star}\).
We expand on the value of this factorisation in the following discussion about a potential-space implementation of \ref{eqn:path-phi-match-update}.

\paragraph{A ``dual'' version of \textsf{path-\(\Phi\)-match}}
We investigate here a ``dual'' potential-space implementation which is based on the factorisation form in \cref{eqn:path-factorise} in contrast to the discussion above which was based on the path measures.
This connects to updates from \ref{eqn:phi-match-update} for eOT, and is similar to the continuous flow of SDE viewpoint in \citet[Sec. 4.4]{karimi2024sinkhorn}.
Let \(\phi^{\star}\) and \(\psi^{\star}\) be the Schr\"{o}dinger potentials associated with the static problem (\cref{eqn:static_SB}) with marginals \(\mu, \nu\), and suppose \((f_{t}^{\star})_{t \geq 0}\) and \((g_{t}^{\star})_{t \geq 0}\) are solutions to the Kolmogorov forward and backward equations under \cref{eqn:ref_process_sde}:
\begin{align*}
\partial_t f^\star_t +\nabla\cdot(u^{\textsf{ref}} f^\star_t) -\Delta f^\star_t=0~, & \quad f^\star_0(x) = e^{\psi^\star(x)}~, \\
\partial_t g^\star_t +u^{\textsf{ref}\top}\nabla g_t^{\star}   + \Delta g^\star_t = 0~, & \quad g^\star_T(y)=e^{\phi^\star(y)}~.
\end{align*}
Then, according to \citet[Thm. 3.4]{leonard2014survey} and the Markov property of \(\bbP^{\textsf{ref}}\), we have
\begin{equation*}
\frac{\rmd\mathbb{P}_{[k,l]}^\star}{\rmd\mathbb{P}_{[k,l]}^{\textsf{ref}}}((X_{t})_{t \in [k,l]})=  f_k^\star(X_k) g_l^\star(X_l) \quad \forall ~k < l \in [0,T]~.
\end{equation*}
Therefore with their dynamics fixed, knowing the two boundary functions $\bm{\phi}_T^\star,\bm{\psi}_0^\star$ (\cref{eqn:path-factorise}) gives us all the information about the optimal path measure $\mathbb{P}^{\star}$.
Taking hints, we keep our sequence of path measures $\{\mathbb{P}^n\}_{n\geq 1}$ in the factorized form involving $\bm{\phi}_T^n,\bm{\psi}_0^n$ alone:
\begin{equation}
\label{eqn:factorize_SB_path}
\frac{\rmd\mathbb{P}^n_{[0, T]}}{\rmd\mathbb{P}^{\textsf{ref}}_{[0, T]}}((X_t^n)_{t\in [0, T]}) = \frac{\mathbb{P}^{n}_{0,T}}{\mathbb{P}_{0,T}^{\textsf{ref}}}(X_0^n,X_T^n)=\exp(\bm{\phi}^n_T(X_T^n)+\bm{\psi}^n_0(X_0^n))~.
\end{equation}
This implies that they always solve the SB problem for their own marginals, and consequently remain within the reciprocal and Markovian class.
If moreover, $\bm{\psi}_0^n = (\bm{\phi}_T^{n})^{+}$ then we always have $\mathbb{P}^{n}_{0}=\mu$. In the lemma below, we derive alternative updates on the SDE drift by leverging
the corresponding dual Schrödinger potentials from the static setting, whose proof can be found in \cref{subsec:proof:SB_dual_formal}.

\begin{lemma}
\label{lem:SB_dual_formal}
Let \(\{\phi^{n}\}_{n \geq 1}\) be the sequence of potentials obtained from \ref{eqn:phi-match-update}.
For any \(n \geq 1\), let  \((g_{t}^{n})_{t \in [0, T]}\) be the solution to the Kolmogorov backward equation
\begin{equation*}
\partial_t g_t^n +\nabla g_t^{n \top} u^{\textsf{ref}}  + \Delta g_t^n = 0, \;\; g_T^n(y)=e^{\phi^n(y)}~.
\end{equation*}
The path measure \(\bbP^{n}\) associated with the following SDE
\begin{equation}
\label{eqn:sde_twisted}
\rmd X_t^n= \left\{u^{\emph{\textsf{ref}}}(X_t^n)+\nabla \log g_t^n(X_t^n)\right\} \rmd t + \sqrt{2} \cdot \rmd B_{t}~, \quad X_0^n\sim\mu
\end{equation}
factorises as \cref{eqn:factorize_SB_path} and converges to $\bbP^\star$ at the same rate as the sequence $\{\phi^n\}_{n \geq 1}$ converges to $\phi^\star$ for the eOT problem.
\end{lemma}

Note that the additional drift in \cref{eqn:sde_twisted} implies that $v_t$ in \cref{eqn:posterior_process_sde} is necessarily a gradient vector field.
Through a Cole-Hopf transform, the Kolmogorov backward equation can be translated to a PDE in \(\phi_{t}^{n} = \log g_{t}^{n}\) instead.
This is also evident from the Feynman-Kac formula, which prescribes that \((\phi_{t}^{n})_{t \in [0, T]}\) can be expressed as \(\phi_t^n(y) = \log \mathbb{E}^{\textsf{ref}} [e^{\phi^n(x_T)}|x_t=y]\) with respect to the the reference process \cref{eqn:ref_process_sde}.

\section{Conclusion}
\label{sec:conclusion}
In this work, we systemically synthesise a variety of viewpoints on algorithms for eOT -- specifically those surrounding the Sinkhorn algorithm.
This synthesis, centered around infinite-dimensional optimisation, leads to a collection of novel methods based on either the primal or dual formulation of the eOT problem, allowing us to go beyond the classical Sinkhorn algorithm and IPF for the path-space / dynamic Schr\"{o}dinger bridge problem.
We also see how the viewpoints contribute to provable guarantees for these methods, which notably are not based on any strict assumptions on the marginals \(\mu, \nu\).
It would be interesting to see how these guarantees can be improved; going past the bounded kernel condition in \cref{thm:kernel-sga-update-rate}, and bounded costs condition to establish a desirable form of smoothness for the semi-dual in \cref{lem:L2nu-smooth}.
While this not only encourages a new theoretical perspective for solving the eOT / Schr\"{o}dinger bridge, we hope that this work also spurs the development of new practical methods other than the Sinkhorn / IPF algorithm and its derivatives for the eOT and Schr\"{o}dinger bridge problems respectively.

\subsubsection*{Acknowledgements}

The authors would like to thank Andre Wibisono and Ashia Wilson for their involvement during earlier stages of this work.

\bibliography{references.bib}

\newpage
\appendix

\setlength{\parindent}{0pt}

\section{Proofs}
\label{sec:proofs}
\subsection{Proofs for the properties of the semi-dual \(J\)}

\cref{lem:concavity-of-J,lem:Linf-smooth,lem:L2nu-smooth} are corollaries of the following more general lemma; its proof is given in \cref{sec:prf:intermediate-semi-dual}.

\begin{lemma}
\label{lem:broader-lem}
Let \(\phi, \overline{\phi} \in L^{1}(\nu)\).
For \(t \in [0, 1]\), define \(\tilde{\phi}_{t} := \phi + t \cdot (\overline{\phi} - \phi)\) and the conditional distribution \(\rho_{t}(.; x)\) whose density is
\begin{equation*}
    \rho_{t}(y; x) := \frac{\exp\left(\tilde{\phi}_{t}(y) - \frac{c(x, y)}{\varepsilon}\right)\nu(y)}{\int_{\calY} \exp\left(\tilde{\phi}_{t}(y') - \frac{c(x, y')}{\varepsilon}\right)\nu(y')\rmd y'}~.
\end{equation*}
Then,
\begin{equation*}
    J(\overline{\phi}) - J(\phi) - \langle \delta J(\phi), \overline{\phi} - \phi\rangle = -\frac{1}{2} \int_{0}^{1} \bbE_{x \sim \mu}\left[\bbV_{\rho_{t}(.; x)}[\overline{\phi} - \phi]\right]~\rmd t~.
\end{equation*}
\end{lemma}

\subsubsection{Proof of \cref{lem:concavity-of-J}}
\begin{proof}
From \cref{lem:broader-lem} and the non-negativity of the variance, we have for any \(\phi, \overline{\phi} \in L^{1}(\nu)\) that
\begin{equation*}
    J(\overline{\phi}) - J(\phi) - \langle \delta J(\phi), \overline{\phi} - \phi\rangle \leq 0~.
\end{equation*}
This completes the proof.
\end{proof}

\subsubsection{Proof of \cref{lem:Linf-smooth}}
\begin{proof}
By definition of the variance
\begin{equation*}
    \bbV_{\rho_{t}(.; x)}[\overline{\phi} - \phi] \leq \int_{\calY} (\overline{\phi}(y) - \phi(y))^{2}~\rho_{t}(y; x)\rmd y \leq \|\overline{\phi} - \phi\|_{L^{\infty}(\calY)}^{2}~.
\end{equation*}
The final inequality is by H\"{o}lder's inequality.
Substituting this in the result of \cref{lem:broader-lem}, we get
\begin{equation*}
    J(\overline{\phi}) - J(\phi) - \langle \delta J(\phi), \overline{\phi} - \phi\rangle \geq -\frac{\|\overline{\phi} - \phi\|^{2}_{L^{\infty}(\calY)}}{2}~.
\end{equation*}
\end{proof}

\subsubsection{Proof of \cref{lem:L2nu-smooth}}
\begin{proof}
From the convexity of \(\calS_{B}\), note that \(\tilde{\phi}_{t} = \phi + t \cdot (\overline{\phi} - \phi) \in \calS_{B}\).
Since \(\calS_{B} \subset L^{1}(\nu)\), we have by \cref{lem:broader-lem} that
\begin{align*}
    J(\overline{\phi}) - J(\phi) - \langle \delta J(\phi), \overline{\phi} - \phi\rangle &= -\frac{1}{2} \int_{0}^{1} \bbE_{x \sim \mu}\left[\bbV_{\rho_{t}(.; x)}[\overline{\phi} - \phi]\right]~\rmd t \\
    &\geq -\frac{1}{2}\int_{0}^{1} \left\{\int_{\calX}\int_{\calY}(\overline{\phi}(y) - \phi(y))^{2} \rho_{t}(y; x) \mu(x)~ \rmd x \rmd y\right\} ~\rmd t~.
\end{align*}
By Fubini's theorem, we can first compute \(\int_{\calX} \rho_{t}(y; x) \mu(x)\rmd x\) and then integrate w.r.t. \(\calY\).
\begin{align*}
    \int_{\calX}\rho_{t}(y; x)\mu(x)\rmd x &= \int_{\calX} \frac{\exp\left(\tilde{\phi}_{t}(y) - \frac{c(x, y)}{\varepsilon}\right) \nu(y)}{\int_{\calY} \exp\left(\tilde{\phi}_{t}(y') - \frac{c(x, y')}{\varepsilon}\right)\nu(y') \rmd y'} \mu(x)\rmd x \\
    &= \bbE_{x \sim \mu}\left[\exp\left(\tilde{\phi}_{t}(y) -\frac{c(x, y)}{\varepsilon}\right) \cdot \bbE_{y' \sim \nu}\left[\exp\left(\tilde{\phi}_{t}(y')-\frac{c(x, y')}{\varepsilon}\right)\right]^{-1}\right] \cdot \nu(y) \\
    &\overset{(a)}\leq \bbE_{(x, y') \sim \mu \otimes \nu}\left[\exp\left(\frac{c(x, y') - c(x, y)}{\varepsilon} + \tilde{\phi}_{t}(y) - \tilde{\phi}_{t}(y')\right)\right] \cdot \nu(y) \\
    &\overset{(b)}\leq \underbrace{e^{2B} \cdot \bbE_{(x, y') \sim \mu \otimes \nu}\left[\exp\left(\frac{c(x, y')}{\varepsilon}\right)\right]}_{\lambda(B)} \cdot~\nu(y)~.
\end{align*}
Step \((a)\) uses Jensen's inequality, and step \((b)\) uses the fact that for \(y, y'\), \(\tilde{\phi}_{t}(y) - \tilde{\phi}_{t}(y') \leq 2B\) for \(y, y'\) almost everywhere.
Substituting this in the result of \cref{lem:broader-lem}, we have
\begin{align*}
    J(\overline{\phi}) - J(\phi) - \langle \delta J(\phi), \overline{\phi} - \phi\rangle &\geq -\frac{1}{2}\int_{0}^{1} \left\{\int_{\calY} (\overline{\phi}(y) - \phi(y))^{2} \cdot \lambda(B) \cdot \nu(y)\rmd y\right\}~\rmd t \\
    &=-\frac{\lambda(B) \cdot \|\overline{\phi} - \phi\|_{L^{2}(\nu)}^{2}}{2}~.
\end{align*}
\end{proof}

\subsubsection{Proof of \cref{lem:broader-lem}}
\label{sec:prf:intermediate-semi-dual}

\begin{proof}
Recall that the first variation of the semi-dual \(J\) is
\begin{equation*}
    \delta J(\phi)(y) = \nu(y) - \pi(\phi, \phi^{+})_{\calY}(y) = \int_{\calX} \nu(y)\mu(x)\rmd x -  \int_{\calX} \frac{\exp\left(\phi(y) - \frac{c(x, y)}{\varepsilon}\right)\nu(y)}{\int_{\calY} \exp\left(\phi(y') - \frac{c(x, y')}{\varepsilon}\right) \nu(y')\rmd y'} \mu(x)\rmd x~.
\end{equation*}
For a fixed \(x \in \calX\), consider the function
\begin{equation*}
    j_{x}(\phi)(y) := \nu(y) -  \frac{\exp\left(\phi(y) - \frac{c(x, y)}{\varepsilon}\right)\nu(y)}{\int_{\calY} \exp\left(\phi(y') - \frac{c(x, y')}{\varepsilon}\right) \nu(y')\rmd y'}~.
\end{equation*}
and hence for any \(\phi, \overline{\phi} \in \calS_{B}\), we have
\begin{equation*}
    j_{x}(\phi)(y) - j_{x}(\overline{\phi})(y) = \frac{\exp\left(\overline{\phi}(y) - \frac{c(x, y)}{\varepsilon}\right)\nu(y)}{\int_{\calY} \exp\left(\overline{\phi}(y') - \frac{c(x, y')}{\varepsilon}\right) \nu(y')\rmd y'} - \frac{\exp\left(\phi(y) - \frac{c(x, y)}{\varepsilon}\right)\nu(y)}{\int_{\calY} \exp\left(\phi(y') - \frac{c(x, y')}{\varepsilon}\right) \nu(y')\rmd y'}
\end{equation*}
Note that \(j_{x}(\phi) - j_{x}(\overline{\phi}) = \rho_{1}(.; x) - \rho_{0}(.; x) = \int_{0}^{1}\dot{\rho_{t}}(.; x) \rmd t\).
We have by direct calculation that
\begin{equation*}
    \frac{\rmd}{\rmd t}\rho_{t}(y; x) \equiv \dot{\rho_{t}}(y; x) = \left\{(\overline{\phi}(y) - \phi(y))- \int_{\calY} (\overline{\phi}(y') - \phi(y')) \rho_{t}(y'; x) \rmd y'\right\}~\rho_{t}(y; x)\rmd y~.
\end{equation*}
Consequently,
\begin{align*}
    \langle j_{x}(\phi) - j_{x}(\overline{\phi}), \phi - \overline{\phi}\rangle &= \int_{\calY} (\phi(y) - \overline{\phi}(y)) \cdot (\rho_{1}(y; x) - \rho_{0}(y; x))~\rmd y \\
    &= \int_{\calY}\int_{0}^{1} (\phi(y) - \overline{\phi}(y)) \cdot \dot{\rho_{t}}(y; x) ~\rmd t\rmd y \\
    &= -\int_{0}^{1} \int_{\calY} (\overline{\phi}(y) - \phi(y))^{2} ~\rho_{t}(y; x) ~\rmd y\rmd t \\
    &\qquad +\int_{0}^{1}\left\{\int_{\calY}(\overline{\phi}(y) - \phi(y)) \cdot \rho_{t}(y; x) \rmd y\right\}^{2}\rmd t\\
    &= -\int_{0}^{1} \bbV_{\rho_{t}(.; x)}[\overline{\phi} - \phi] ~\rmd t~.\numberthis\label{eqn:temp_bound}
\end{align*}
Taking the expectation w.r.t. \(\mu\) on both sides and by Fubini's theorem, we have
\begin{equation}
    \langle \delta J(\phi) - \delta J(\overline{\phi}), \phi - \overline{\phi}\rangle = -\int_{0}^{1} \bbE_{x \sim \mu}\left[\bbV_{\rho_{t}(.; x)}[\overline{\phi} - \phi]\right]~\rmd t~.
\end{equation}
Define \(\tilde{J}_{t} = J(\tilde{\phi}_{t}) - \langle \delta J(\phi), \tilde{\phi}_{t}\rangle\).
By the chain rule,
\begin{equation*}
    \dot{\tilde{J}}_{t} = \langle \delta J(\tilde{\phi}_{t}), \overline{\phi} - \phi\rangle - \langle \delta J(\phi), \overline{\phi} - \phi\rangle = \langle \delta J(\tilde{\phi}_{t}) - \delta J(\phi), \overline{\phi} - \phi\rangle~.
\end{equation*}
By the fundamental theorem of calculus,
\begin{align*}
    J(\overline{\phi}) - J(\phi) - \langle \delta J(\phi), \overline{\phi} - \phi\rangle &= J_{1} - J_{0} \\
    &= \int_{0}^{1} \dot{\tilde{J}}_{s}~\rmd s \\
    &= \int_{0}^{1} \langle \delta J(\tilde{\phi}_{s}) - \delta J(\phi), \overline{\phi} - \phi\rangle~\rmd s \\
    &= \int_{0}^{1} \frac{1}{s} \cdot \langle \delta J(\tilde{\phi}_{s}) - \delta J(\phi), \tilde{\phi}_{s} - \phi\rangle~\rmd s \\
    &= -\int_{0}^{1} \frac{1}{s} \cdot \int_{0}^{1} \bbE_{x \sim \mu}\left[\bbV_{\rho_{t}(.; x)}[\tilde{\phi}_{s} - \phi]\right] \rmd t ~\rmd s \\
    &= -\int_{0}^{1}\int_{0}^{1} s \cdot \bbE_{x \sim \mu}\left[\bbV_{\rho_{t}(.; x)}[\overline{\phi} - \phi]\right] \rmd t ~\rmd s \\
    &= -\frac{1}{2} \int_{0}^{1} \bbE_{x \sim \mu}\left[\bbV_{\rho_{t}(.; x)}[\overline{\phi} - \phi]\right]~\rmd t~.
\end{align*}
\end{proof}

\subsection{Proofs for the statements in \cref{sec:measure_interpretation}}
\subsubsection{Proof of \cref{lem:KL_primal_projection}}
\label{sec:prf:KL_primal_projection}
\begin{proof}
Consider some \(n \geq 0\).
By the decomposition of KL divergence,
\begin{equation*}
    \KLdist(\pi \| \pi^{n}) = \bbE_{y \sim \pi_{\calY}}[\KLdist(\pi_{\calX | \calY}(. | y) \| \pi^{n}_{\calX | \calY}(. | y))] + \KLdist(\pi_{\calY} \| \pi^{n}_{\calY})~.
\end{equation*}
For convenience, we denote \(\projY(\pi^{n}; \Phi)\) as \(\pi^{n + \nicefrac{1}{2}}\).
By definition of \(\projY(\pi^{n}; \Phi)\), we have
\begin{equation*}
    \pi^{n + \nicefrac{1}{2}}_{\calY}(y) = \frac{1}{Z} \cdot \pi^{n}_{\calY}(y) \cdot \frac{\Phi(\nu)(y)}{\Phi(\pi^{n}_{\calY})(y)}; \quad \pi_{\calX | \calY}^{n + \nicefrac{1}{2}}(x | y) = \pi^{n}_{\calX | \calY}(x | y)~.
\end{equation*}
Above, \(Z = \bbE_{\sfy \sim \pi^{n}_{\calY}}\left[\frac{\Phi(\nu)(\sfy)}{\Phi(\pi^{n}_{\calY})(\sfy)}\right]\).
Therefore,
\begin{equation*}
    \pi^{n + \nicefrac{1}{2}}(x, y) = \frac{1}{Z} \cdot \pi^{n}(x, y) \cdot \frac{\Phi(\nu)(y)}{\Phi(\pi_{\calY}^{n})(y)}~.
\end{equation*}
Since \(\pi^{n} = \pi(\phi^{n}, (\phi^{n})^{+})\), this shows that \(\pi^{n + \nicefrac{1}{2}}\) factorises as
\begin{equation*}
    \pi^{n + \nicefrac{1}{2}}(x, y) = \exp\left(-\psi^{n + \nicefrac{1}{2}}(x) + \phi^{n + \nicefrac{1}{2}}(y) - \frac{c(x, y)}{\varepsilon}\right)\mu(x)\nu(y)~
\end{equation*}
where
\(\phi^{n + \nicefrac{1}{2}}(y) = \phi^{n}(y) + (\log \Phi(\nu)(y) - \log \Phi(\pi^{n}_{\calY})(y))\) and \(\psi^{n + \nicefrac{1}{2}}(x) = \psi^{n}(x) + \log Z\)~.
From \cite[Corr. B.1]{karimi2024sinkhorn}, we have that \(\projX(\pi^{n + \nicefrac{1}{2}}, \pi^{n}; \eta)\) satisfies
\begin{align*}
    \projX(\pi^{n + \nicefrac{1}{2}}, \pi^{n}; \eta)_{\calY | \calX}(y | x) &= \frac{\pi^{n + \nicefrac{1}{2}}_{\calY | \calX}(y | x)^{\eta} \cdot \pi^{n}_{\calY | \calX}(y | x)^{1 - \eta}}{C(x)} \\
    C(x) &= \int_{\calY} \pi^{n + \nicefrac{1}{2}}_{\calY | \calX}(y | x)^{\eta} \cdot \pi^{n}_{\calY | \calX}(y | x)^{1 - \eta} \rmd y~.
\end{align*}
The factorisations of \(\pi^{n}\) and \(\pi^{n + \nicefrac{1}{2}}\) results in \(\projX(\pi^{n + \nicefrac{1}{2}}, \pi^{n}; \eta)\) factorising as
\begin{equation*}
    \projX(\pi^{n + \nicefrac{1}{2}}, \pi^{n}; \eta)(x, y) = \exp\left(\bar{\phi}(y) - \bar{\psi}(x) - \frac{c(x, y)}{\varepsilon}\right)\mu(x)\nu(y)
\end{equation*}
and specifically,
\begin{align}
    \bar{\phi}(y) &= \eta \cdot \phi^{n + \nicefrac{1}{2}}(y) + (1 - \eta) \cdot \phi^{n}(y) \nonumber \\
    &= \phi^{n}(y) + \eta \cdot (\log \Phi(\nu)(y) - \log \Phi(\pi^{n}_{\calY})(y))~.\label{eqn:itermiediate}
\end{align}
Since \(\projX(\pi^{n + \nicefrac{1}{2}}, \pi^{n}; \eta)_{\calX} = \mu\), this implies
\begin{equation*}
    \bar{\psi}(x) = \log \int_{\calY} \exp\left(\bar{\phi}(y) - \frac{c(x, y)}{\varepsilon}\right)\nu(y)\rmd y = \bar{\phi}^{+}(x)~.
\end{equation*}
Hence, comparing \cref{eqn:itermiediate} with \ref{eqn:phi-match-update} we have \(\projX(\pi^{n + \nicefrac{1}{2}}, \pi^{n}; \eta) = \pi(\phi^{n + 1}; (\phi^{n + 1})^{+})\) which completes the proof.
\end{proof}

\subsubsection{Proof of \cref{lem:proximal-point}}
\label{sec:prf:proximal-point}
\begin{proof}
For convenience, we use the shorthand notation \(\tilde{\pi} = \projY(\pi; \Phi)\).
As in the proof of \cref{lem:KL_primal_projection}, \cref{eqn:measure_1} ensures that
\begin{gather*}
    \tilde{\pi}(x, y) = \frac{1}{Z} \cdot \pi(x, y) \cdot \frac{\Phi(\nu)(y)}{\Phi(\pi_{\calY})(y)}~; \qquad Z = \bbE_{\sfy \sim \pi_{\calY}}\left[\frac{\Phi(\nu)(\sfy)}{\Phi(\pi_{\calY})(\sfy)}\right] \\
    \Rightarrow \log \Phi(\pi_{\calY})(y) - \log \Phi(\nu)(y) = \log \frac{\pi(x, y)}{\tilde{\pi}(x, y)} - \log Z~.
\end{gather*}
The objective in \cref{eqn:proximal-point} with \(\calF \leftarrow \calV_{\Phi}\) can be simplified as
\begin{align*}
    \langle \calV_{\Phi}(\pi), \bar{\pi} - \pi\rangle &+ \frac{1}{\eta} \cdot \KLdist(\bar{\pi} \| \pi) 
    \\
    &\iint \calV_{\Phi}(\pi)(x, y)(\bar{\pi}(x, y) - \pi(x, y))\rmd x \rmd y \\
    &\qquad + \frac{1}{\eta} \cdot \iint \bar{\pi}(x, y) \log\left(\frac{\bar{\pi}(x, y)}{\pi(x, y)}\right)\rmd x \rmd y \\
    &= \iint \left(\log \Phi(\pi^{n}_{\calY})(y) - \log \Phi(\nu)(y)\right) \cdot \bar{\pi}(x, y) \rmd x \rmd y + \log Z \\
    &\quad + \frac{1}{\eta} \cdot \iint \bar{\pi}(x, y) \log\left(\frac{\bar{\pi}(x, y)}{\pi(x, y)}\right)\rmd x \rmd y \\
    &\qquad - \underbrace{\left(\log Z + \iint \left(\log \Phi(\pi_{\calY})(y) - \log \Phi(\nu)(y)\right) \cdot \pi(x, y) \rmd x \rmd y\right)}_{c(\pi)} \\
    &= \iint \bar{\pi}(x, y) \cdot \log\left(\frac{\pi(x, y)}{\tilde{\pi}(x, y)} \right) \rmd x \rmd y \\
    &\qquad + \frac{1}{\eta} \cdot \iint \bar{\pi}(x, y) \log\left(\frac{\bar{\pi}(x, y)}{\pi(x, y)}\right) \rmd x \rmd y + c(\pi) \\
    &= \frac{1}{\eta} \left\{ \iint \bar{\pi}(x, y) \cdot \log\left[\left(\frac{\bar{\pi}(x, y)}{\tilde{\pi}(x, y)}\right)^{\eta}\left(\frac{\bar{\pi}(x, y)}{\pi(x, y)}\right)^{1 - \eta}\right] \rmd x \rmd y\right\}\\
    &\qquad + c(\pi)~.
\end{align*}
The objective in \(\projX(\tilde{\pi}, \pi; \eta)\) can be expanded as
\begin{equation*}
    \eta \KLdist(\bar{\pi} \| \tilde{\pi}) + (1 - \eta) \KLdist(\bar{\pi} \| \pi) = \iint \bar{\pi}(x, y) \cdot \log\left[\left(\frac{\bar{\pi}(x, y)}{\tilde{\pi}(x, y)}\right)^{\eta}\left(\frac{\bar{\pi}(x, y)}{\pi(x, y)}\right)^{1 - \eta}\right] \rmd x \rmd y~
\end{equation*}
thus establishing the equivalence in the statement as \(\projX\) also minimises over the set \(\{\overline{\pi} : \overline{\pi}_{\calX} = \mu\}\) and \(c(\pi)\) is a constant.
\end{proof}

\subsubsection{Proof of \cref{thm:kernel-sga-update-rate}}
\label{sec:prf:kernel-sga-update-rate}
\begin{proof}

The proof is be obtained in the manner of the proof of \citet[Thm. 4]{aubin2022mirror} while catering to the squared MMD \(\calL_{k}\). We give the details here for completeness.

For an arbitrary \(n \geq 0\), we have the following identity for any \(\overline{\pi}\) such that \(\overline{\pi}_{\calX} = \mu\) that 
\begin{multline}
\label{eqn:three-point}
    \eta \cdot \langle \calV_{\Phi_{k}}(\pi^{n}), \overline{\pi} - \pi^{n}\rangle + \KLdist(\overline{\pi} \| \pi^{n}) \\
    \geq \eta \cdot \langle \calV_{\Phi_{k}}(\pi^{n}), \pi^{n + 1} - \pi^{n}\rangle + \KLdist(\pi^{n + 1} \| \pi^{n}) + \KLdist(\overline{\pi} \| \pi^{n + 1})~.
\end{multline}
This is obtained by the three-point identity \cite[Lem. 3]{aubin2022mirror} with
\begin{equation*}
    C \leftarrow \{\pi \in \calP(\calX \times \calY) : \pi_{\calX} = \mu\}~, \quad \calG \leftarrow \eta \cdot \langle \calV_{\Phi_{k}}(\pi^{n}), \cdot - \pi^{n}\rangle~, \quad D_{\phi}(\cdot | \cdot)\leftarrow \KLdist(\cdot \| \cdot)~.
\end{equation*}
By the definition of \(\calV_{\Phi_{k}}(\pi^{n}) = \frakm_{k}(\pi^{n}_{\calY}) - \frakm_{k}(\nu) = \delta \calL_{k}(\pi^{n}_{\calY}; \nu)\), we have
\begin{align*}
    \langle\calV_{\Phi_{k}}(\pi^{n}), \pi^{n + 1} - \pi^{n}\rangle &= \int_{\calY}\int_{\calX} \delta\calL_{k}(\pi^{n}_{\calY}; \nu)(y) \cdot (\pi^{n+ 1}(x, y) - \pi^{n}(x, y)) ~\rmd x\rmd y \\
    &= \langle \delta \calL_{k}(\pi^{n}_{\calY}; \nu), \pi^{n + 1}_{\calY} - \pi^{n}_{\calY}\rangle~.
\end{align*}

From \cref{prop:aubin-smooth}, we know that that in this case
\begin{align*}
    \calL_{k}(\pi_{\calY}^{n + 1}; \nu) &\leq \calL_{k}(\pi_{\calY}^{n}; \nu) + \langle \delta \calL_{k}(\pi^{n}_{\calY}; \nu), \pi^{n + 1}_{\calY} - \pi^{n}_{\calY}\rangle + 2c_{k} \cdot \KLdist(\pi^{n + 1}_{\calY} \| \pi^{n}_{\calY}) \\
    &= \calL_{k}(\pi^{n}_{\calY}; \nu) + \langle \calV_{\Phi}(\pi^{n}), \pi^{n + 1} - \pi^{n}\rangle + 2c_{k} \cdot \KLdist(\pi^{n + 1}_{\calY} \| \pi^{n}_{\calY}) \\
    &\overset{(a)}\leq \calL_{k}(\pi_{\calY}^{n}; \nu) + \langle \calV_{\Phi}(\pi^{n}), \pi^{n + 1} - \pi^{n}\rangle + 2c_{k} \cdot \KLdist(\pi^{n + 1} \| \pi^{n}) \numberthis\label{eqn:aubin-smooth-consequence}\\
    &\overset{(b)}\leq \calL_{k}(\pi_{\calY}^{n}; \nu) + \left(2c_{k} - \frac{1}{\eta}\right) \cdot \KLdist(\pi^{n + 1} \| \pi^{n}) - \frac{1}{\eta} \cdot \KLdist(\pi^{n} \| \pi^{n + 1}) \\
    &\overset{(c)}\leq \calL_{k}(\pi_{\calY}^{n}; \nu)\numberthis\label{eqn:MMD-decrease}~.
\end{align*}
Step \((a)\) is a consequence of the KL decomposition, step \((b)\) applies \cref{eqn:three-point} for \(\overline{\pi} \leftarrow \pi^{n}\), and step \((c)\) uses the fact that \(\eta = \min\left\{\frac{1}{2c_{k}}, 1\right\} \leq \frac{1}{2c_{k}}\) and the non-negativity of the KL divergence.
We also have by \cref{prop:aubin-smooth} that
\begin{align*}
    \calL_{k}(\overline{\pi}_{\calY}; \nu) -\calL_{k}(\pi_{\calY}^{n}; \nu) &\geq \langle \delta \calL_{k}(\pi^{n}_{\calY}; \nu), \overline{\pi}_{\calY} - \pi^{n}_{\calY}\rangle \\
    & = \langle \calV_{\Phi}(\pi^{n}), \overline{\pi} - \pi^{n}\rangle~.
\end{align*}
Substituting the above and \cref{eqn:aubin-smooth-consequence} in \cref{eqn:three-point} with \(\eta = \min\left\{\frac{1}{2c_{k}}, 1\right\}\), we obtain
\begin{equation*}
    \frac{1}{\max\{2c_{k}, 1\}}\calL_{k}(\pi^{n + 1}_{\calY}; \nu) - \frac{1}{\max\{2c_{k}, 1\}}\calL_{k}(\overline{\pi}_{\calY}; \nu) \leq \KLdist(\overline{\pi} \| \pi^{n}) - \KLdist(\overline{\pi} \| \pi^{n + 1})~.
\end{equation*}
Summing both sides from \(n = 0\) to \(n = N - 1\) yields
\begin{equation*}
    \frac{1}{\max\{2c_{k}, 1\}}\sum_{n = 1}^{N}\{\calL_{k}(\pi_{\calY}^{n}; \nu) - \calL_{k}(\overline{\pi}_{\calY}; \nu)\} \leq \KLdist(\overline{\pi} \| \pi^{0}) - \KLdist(\overline{\pi} \| \pi^{N})~.
\end{equation*}
We know that \(\calL_{k}(\pi^{n + 1}_{\calY}; \nu) \leq \calL_{k}(\pi^{n}_{\calY}; \nu)\) from \cref{eqn:MMD-decrease}.
Noting that \(\pi^{\star}\) satisfies \(\pi^{\star}_{\calX} = \mu\) and \(\pi^{\star}_{\calY} = \nu\), we substitute \(\overline{\pi} \leftarrow \pi^{\star}\) above and this leads to
\begin{equation*}
    \calL_{k}(\pi^{N}_{\calY};\nu) \leq \frac{\max\{2c_{k}, 1\}}{N} \cdot \KLdist(\pi^{\star} \| \pi^{0})~.
\end{equation*}
\end{proof}

\subsubsection{Proof of \cref{lem:md_interpret}}
\label{sec:prf:md_interpret}
\begin{proof}
We begin by writing the definition of \(\delta \varphi(\pi^{n})\)
\begin{equation*}
    \delta \varphi(\pi^{n}) = \log \frac{\pi^{n}}{\pi^{\mathrm{ref}}_{\varepsilon}} = (-(\phi^{n})^{+}) \oplus \phi^{n}~.
\end{equation*}
By definition of \(\calV_{\Phi}\) and \ref{eqn:phi-match-update},
\begin{equation*}
    \delta \varphi(\pi^{n}) - \eta \cdot \calV_{\Phi}(\pi^{n}) = (-(\phi^{n})^{+}) \oplus (\phi^{n} - \eta \cdot (\log \Phi(\pi^{n}_{\calY}) - \log \Phi(\nu))) = (-(\phi^{n})^{+}) \oplus \phi^{n + 1}~.
\end{equation*}
By direct calculation, we see that \(\delta \varphi^{\star}(f \oplus g) \in \calQ\):
\begin{align*}
    \delta \varphi^{\star}(f \oplus g)(x, y) &= \exp\left(g(y) - \frac{c(x, y)}{\varepsilon}\right) \mu(x)\nu(y) \cdot \left(\int_{\calY} \exp\left(g(y) - \frac{c(x, y)}{\varepsilon}\right)\nu(y)\rmd y\right)^{-1} \\
    &= \exp\left(g(y) - g^{+}(x) - \frac{c(x, y)}{\varepsilon}\right)\mu(x)\nu(y) \\
    &= \pi(g, g^{+})~.
\end{align*}
Applying the mapping $\delta \varphi^{*}$  to \(\delta \varphi(\pi^{n}) - \eta \cdot \calV_{\Phi}(\pi^{n})\) and the above identity, we have
\begin{equation*}
    \delta \varphi^{*}(\delta \varphi(\pi^{n}) - \eta \cdot \calV_{\Phi}(\pi^{n})) = \delta \varphi^{*}((-(\phi^{n})^{+}) \oplus \phi^{n + 1}) = \delta \varphi^{*}(0 \oplus \phi^{n + 1}) = \pi(\phi^{n+ 1}, (\phi^{n + 1})^{+}) = \pi^{n + 1}~.
\end{equation*}
showing the equivalence to \ref{eqn:phi-match-update}.
\end{proof}

\subsection{Proofs of theorems in \cref{sec:extensions}}

\subsubsection{Proof of \cref{thm:sign-sga-conv-rate}}
\label{app:sec:prf:sign-sga-conv-rate}

\begin{proof}
From \cref{lem:Linf-smooth}, we have for any \(\phi, \overline{\phi} \in L^{1}(\nu) \cap L^{\infty}(\calY)\) that
\begin{align*}
    J(\overline{\phi}) - J(\phi) - \langle \delta J(\phi), \overline{\phi} - \phi\rangle &= \int_{0}^{1} \langle \delta J(\phi + t (\overline{\phi} - \phi)) - \delta J(\phi), \overline{\phi} - \phi\rangle \rmd t \\
    &\geq -\frac{\|\overline{\phi} - \phi\|^{2}_{\infty}}{2}~.
\end{align*}

For any \(n \geq 0\), substituting \(\phi \leftarrow \phi^{n}\) and \(\overline{\phi} \leftarrow \phi^{n + \nicefrac{1}{2}} = \sfM^{\textsf{sign-SGA}}(\phi^{n}; 1)\), we get
\begin{align*}
    J(\phi^{n + \nicefrac{1}{2}}) &\geq J(\phi^{n}) + \langle \delta J(\phi^{n}), \phi^{n + \nicefrac{1}{2}} - \phi^{n}\rangle - \frac{\|\phi^{n + \nicefrac{1}{2}} - \phi^{n}\|_{\infty}^{2}}{2} \\
    &\overset{(a)}= J(\phi^{n}) + \|\delta J(\phi^{n})\|_{L^{1}(\calY)} \cdot \|\delta J(\phi^{n})\|_{L^{1}(\calY)} - \frac{1}{2} \cdot \|\delta J(\phi^{n})\|^{2}_{L^{1}(\calY)}~\\
    &= J(\phi^{n}) + \frac{\|\delta J(\phi^{n})\|_{L^{1}(\calY)}^{2}}{2}~.\numberthis\label{eqn:one-step-descent-sign-sga}
\end{align*}
Step \((a)\) above is due to the fact that \(\langle \mathrm{sign}(\delta J(\phi^{n})), \delta J(\phi^{n})\rangle = \|\delta J(\phi^{n})\|_{1}\).
By the shift invariance of the semi-dual, \(J(\phi^{n + 1}) = J(\phi^{n + \nicefrac{1}{2}})\), which implies
\begin{equation*}
    J(\phi^{n + 1}) \geq J(\phi^{n}) + \frac{\|\delta J(\phi^{n})\|^{2}_{L^{1}(\calY)}}{2}~.
\end{equation*}
Hence \(\phi^{n + 1} \in \calT_{\phi^{0}, y_{\texttt{anc}}}\) as \(\phi^{n + 1}(y_{\texttt{anc}}) = \phi^{n}(y_{\texttt{anc}})\).
Next, by concavity of \(J\) (\cref{lem:concavity-of-J}) that
\begin{equation}
\label{eqn:concave-sign-sga}
    J(\widetilde{\phi}^{\star}) \leq J(\phi^{n}) + \langle \delta J(\phi^{n}), \widetilde{\phi}^{\star} - \phi^{n}\rangle~.
\end{equation}
Define the Lyapunov function \(E_{n} := \frac{n(n + 1)}{2} \cdot (J(\phi^{n}) - J(\widetilde{\phi}^{\star}))\).
We have
\begin{align*}
    E_{n + 1} - E_{n} &= \frac{(n + 2)(n + 1)}{2} \cdot (J(\phi^{n + 1}) - J(\phi^{n})) + (n + 1) \cdot (J(\phi^{n}) - J(\widetilde{\phi}^{\star})) \\
    &\overset{(a)}\geq (n + 1) \cdot \left\{\frac{n + 2}{4} \cdot \|\delta J(\phi^{n})\|_{1}^{2} + \langle \delta J(\phi^{n}), \phi^{n} - \widetilde{\phi}^{\star}\rangle\right\} \\
    &\overset{(b)}\geq -\frac{n + 1}{n + 2} \cdot \|\phi^{n} - \widetilde{\phi}^{\star}\|_{\infty}^{2} \\
    &\overset{(c)}\geq -\mathrm{diam}(\calT_{\phi^{0}, y_{\texttt{anc}}}; L^{\infty}(\calY))^{2}~.
\end{align*}
Above, step \((a)\) applies \cref{eqn:one-step-descent-sign-sga,eqn:concave-sign-sga}, and step \((b)\) applies the H\"{o}lder-Young inequality.
Finally, step \((c)\) uses the fact that \(\phi^{n}, \widetilde{\phi}^{\star} \in \calT_{\phi^{0}, y_{\texttt{anc}}}\) shown previously.
Summing the above inequality from \(n = 0\) to \(n = N -1\), we get
\begin{align*}
    E_{N} - E_{0} &\geq -N \cdot \mathrm{diam}(\calT_{\phi^{0}, y_{\texttt{anc}}}; L^{\infty}(\calY))^{2} \\
    \Rightarrow J(\phi^{N}) - J(\widetilde{\phi}^{\star}) &\geq - \frac{2 \cdot \mathrm{diam}(\calT_{\phi^{0}, y_{\texttt{anc}}}; L^{\infty}(\calY))^{2}}{N + 1}~.
\end{align*}
\end{proof}

\subsubsection{Proof of \cref{thm:proj-sga-conv-rate}}
\label{app:sec:prf:proj-sga-conv-rate}

Before we give the proof, we lay out some preliminaries and intermediate results that will come in handy to prove \cref{thm:acc-proj-sga-conv-rate} later.

The truncated quadratic approximation to \(J\) centered at \(\phi \in L^{2}(\nu)\) that \ref{eqn:proj-sga-update} is based on:
\begin{equation*}
    \widetilde{J}_{\eta, \calS_{B}}(\overline{\phi}; \phi) := J(\phi) + \left\langle \frac{\delta J(\phi)}{\nu}, \overline{\phi} - \phi \right\rangle_{L^{2}(\nu)} - \frac{1}{2\eta}\|\overline{\phi} - \phi\|_{L^{2}(\nu)}^{2} - \bbI_{\calS_{B}}(\overline{\phi})~.
\end{equation*}
Above, \(\bbI_{\calS_{B}}\) is the convex indicator for \(\calS_{B}\) which evaluates to \(0\) if \(\overline{\phi} \in \calS_{B}\) and \(\infty\) otherwise.
Note that
\begin{equation*}
    \widetilde{J}_{\eta, \calS_{B}}(\overline{\phi}; \phi) = J(\phi) + \frac{\eta}{2} \cdot \left\|\frac{\delta J(\phi)}{\nu}\right\|_{L^{2}(\nu)}^{2} - \frac{1}{2\eta} \cdot \left\|\overline{\phi} - \left(\phi + \eta \cdot \frac{\delta J(\phi)}{\nu}\right)\right\|_{L^{2}(\nu)}^{2} - \bbI_{\calS_{B}}(\overline{\phi})~.
\end{equation*}
As a result, we have the alternate characterisation of \(\sfM^{\textsf{proj-SGA}}_{\calS_{B}}\) as
\begin{equation*}
    \sfM^{\textsf{proj-SGA}}_{\calS_{B}}(\phi; \eta) = \argmax_{\overline{\phi} \in \calS_{B}}\; \widetilde{J}_{\eta, \calS_{B}}(\overline{\phi}; \phi)~.
\end{equation*}
We use \(\overline{J}_{\calS_{B}}\) to denote the composite function \(J + \bbI_{\calS_{B}}\).
Also recall that \(\calS_{B} = \{\phi \in L^{2}(\nu) : \|\phi\|_{L^{\infty}(\calY)} \leq B\}\).

\begin{lemma}
\label{lem:lower-bound-update}
Let \(\phi \in \calS_{\bar{B}}\).
Then, for \(\eta \leq \lambda(\max\{B, \bar{B}\})^{-1}\),
\begin{equation*}
    \overline{J}_{\calS_{B}}(\sfM^{\emph{\textsf{proj-SGA}}}_{\calS_{B}}(\phi; \eta)) \geq \widetilde{J}_{\eta, \calS_{B}}(\sfM^{\emph{\textsf{proj-SGA}}}_{\calS_{B}}(\phi; \eta); \phi)~.
\end{equation*}
\end{lemma}

\begin{lemma}
\label{lem:lem2.3-bt09-general}
Let \(\phi \in \calS_{\bar{B}}\).
For any \(\overline{\phi} \in L^{2}(\nu)\) and \(\eta \leq \lambda(\max\{B, \bar{B}\})^{-1}\), we have that
\begin{equation*}
    \overline{J}_{\calS_{B}}(\sfM^{\textsf{proj-SGA}}_{\calS_{B}}(\phi; \eta)) - \overline{J}_{\calS_{B}}(\overline{\phi})
    \geq \frac{1}{2\eta} \cdot \|\sfM^{\emph{\textsf{proj-SGA}}}_{\calS_{B}}(\phi; \eta) - \phi\|_{L^{2}(\nu)}^{2}
    + \frac{1}{\eta} \cdot \langle \sfM^{\emph{\textsf{proj-SGA}}}_{\calS_{B}}(\phi; \eta) - \phi, \phi - \overline{\phi}\rangle_{L^{2}(\nu)}~.
\end{equation*}
\end{lemma}

\begin{proof}[Proof of \cref{thm:proj-sga-conv-rate}]
For \(\phi^{0} \in \calS_{B}\), each step according to \ref{eqn:proj-sga-update} ensures that \(\phi^{n} \in \calS_{B}\) for all \(n \geq 1\).
For \(\eta \leq \frac{1}{\lambda(B)}\), we have from \cref{lem:lem2.3-bt09-general} applied to \(\overline{\phi} \leftarrow \widetilde{\phi}^{\star}\) and \(\phi \leftarrow \phi^{n}\) for an arbitrary \(n \geq 0\) that
\begin{align*}
    \overline{J}_{\calS_{B}}(\phi^{n + 1}) - \overline{J}_{\calS_{B}}(\widetilde{\phi}^{\star}) &\geq \frac{1}{2\eta} \cdot \|\phi^{n + 1} - \phi^{n}\|_{L^{2}(\nu)}^{2} + \frac{1}{\eta} \cdot \langle \phi^{n + 1} - \phi^{n}, \phi^{n} - \widetilde{\phi}^{\star}\rangle_{L^{2}(\nu)}~\\
    &= \frac{1}{2\eta} \cdot \|\phi^{n + 1} - \widetilde{\phi}^{\star}\|_{L^{2}(\nu)}^{2} - \frac{1}{2\eta} \cdot \|\phi^{n} - \widetilde{\phi}^{\star}\|_{L^{2}(\nu)}^{2}~.
\end{align*}
Summing both sides from \(n = 0\) to \(n = N - 1\) for \(N \geq 1\) we get
\begin{equation*}
    \sum_{n = 0}^{N - 1}(\overline{J}_{\calS_{B}}(\phi^{n + 1}) - \overline{J}_{\calS_{B}}(\widetilde{\phi}^{\star})) \geq \frac{1}{2\eta} \cdot \|\phi^{N} - \widetilde{\phi}^{\star}\|^{2}_{L^{2}(\nu)} - \frac{1}{2\eta} \cdot \|\phi^{0} - \widetilde{\phi}^{\star}\|_{L^{2}(\nu)}^{2}~.
\end{equation*}
Additionally from \cref{lem:lower-bound-update}, we have for the choice of \(\eta\),
\begin{equation*}
    \overline{J}_{\calS_{B}}(\phi^{n + 1}) \geq \widetilde{J}_{\eta, \calS_{B}}(\phi^{n + 1}; \phi^{n}) \geq \widetilde{J}_{\eta, \calS_{B}}(\phi^{n}; \phi^{n}) = \overline{J}_{\calS_{B}}(\phi^{n})~.
\end{equation*}
Hence,
\begin{equation*}
    N \cdot (\overline{J}_{\calS_{B}}(\phi^{N}) - \overline{J}_{\calS_{B}}(\widetilde{\phi}^{\star})) \geq -\frac{1}{2\eta} \cdot \|\phi^{0} - \widetilde{\phi}^{\star}\|_{L^{2}(\nu)}^{2}~.
\end{equation*}
Since \(\phi^{n} \in \calS_{B}\) for all \(n \geq 0\), \(\overline{J}_{\calS_{B}}(\phi^{n}) = J(\phi^{n})\).
\end{proof}

\subsubsection{Proof of \cref{thm:acc-proj-sga-conv-rate}}
\label{app:sec:prf:acc-proj-sga-conv-rate}

Prior to stating the proof for \cref{thm:acc-proj-sga-conv-rate}, we first make the following observations about the sequence \(\{\overline{\phi^{n}}\}_{n \geq 0}\) and \(\{t_{n}\}_{n \geq 1}\) generated by \ref{eqn:acc-proj-sga-update}.
These are: 
\begin{itemize}[leftmargin=*]
    \item for every \(n \geq 0\), \(\overline{\phi}^{n} \in \calS_{B}\), and
    \item for every \(n \geq 1\), \(\frac{t_{n} - 1}{t_{n + 1}} \in (0, 1)\) (\cref{lem:tn-recursion}).
\end{itemize}

A key step towards the proof of \cref{thm:acc-proj-sga-conv-rate} is the following lemma, analogous to \cite[Lem. 4.1]{beck2009fast}.
\begin{lemma}
\label{lem:key-identity}
Let \(\{\overline{\phi}^{n}\}_{n \geq 1}\) be obtained from \ref{eqn:acc-proj-sga-update}.
Define \(v_{n} = \overline{J}(\phi^{\star}) - \overline{J}(\overline{\phi}^{n})\) and \(u_{n} = t_{n} \cdot \overline{\phi}^{n} - (t_{n} - 1) \cdot \overline{\phi}^{n - 1} - \widetilde{\phi}^{\star}\).
Then,
\begin{equation*}
    \frac{2}{\lambda(3B)} \cdot (t_{n}^{2}v_{n} - t_{n + 1}^{2}v_{n + 1}) \geq \|u_{n + 1}\|_{L^{2}(\nu)}^{2} - \|u_{n}\|_{L^{2}(\nu)}^{2}~.
\end{equation*}
\end{lemma}

\begin{proof}[Proof of \cref{thm:acc-proj-sga-conv-rate}]
Since \(\lambda(3B) \geq \lambda(B)\) and \(\phi^{1}, \overline{\phi}^{1} \in \calS_{B}\), \cref{lem:lem2.3-bt09-general} with \(\phi \leftarrow \phi^{1}, \overline{\phi} \leftarrow \widetilde{\phi}^{\star}\) gives
\begin{align*}
    \overline{J}(\overline{\phi}^{1}) - \overline{J}(\phi^{\star}) &\geq \frac{\lambda(3B)}{2} \cdot \|\overline{\phi}^{1} - \phi^{1}\|_{L^{2}(\nu)}^{2} + \lambda(3B) \cdot \langle \overline{\phi}^{1} - \phi^{1}, \phi^{1} - \widetilde{\phi}^{\star}\rangle_{L^{2}(\nu)} \\
    &= \frac{\lambda(3B)}{2} \cdot \left\{\|\overline{\phi}^{1} - \phi^{\star}\|_{L^{2}(\nu)}^{2} - \|\phi^{1} - \phi^{\star}\|_{L^{2}(\nu)}^{2}\right\}~.
\end{align*}
In the notation of \cref{lem:key-identity},
\begin{equation}
\label{eqn:first-iterate}
    -v_{1} \geq \frac{\lambda(3B)}{2} \cdot \|u_{1}\|_{L^{2}(\nu)}^{2} - \frac{\lambda(3B)}{2} \cdot \|\overline{\phi}^{0} - \widetilde{\phi}^{\star}\|_{L^{2}(\nu)}^{2}~.
\end{equation}
Telescoping the identity from \cref{lem:key-identity} for \(n = 1\) to \(N - 1\) gives
\begin{equation*}
    \frac{2}{\lambda(3B)} \cdot (t_{1}^{2}v_{1} - t_{N}^{2}v_{N}) \geq \|u_{N}\|^{2}_{L^{2}(\nu)} - \|u_{1}\|_{L^{2}(\nu)}^{2} \geq -\|u_{1}\|_{L^{2}(\nu)}^{2}~.
\end{equation*}
Rearranging the terms, we have
\begin{equation*}
    v_{N}t_{N}^{2} \leq \frac{\lambda(3B)}{2} \cdot \|u_{1}\|_{L^{2}(\nu)}^{2} + t_{1}^{2}v_{1} \leq \frac{\lambda(3B)}{2} \cdot \|\overline{\phi}^{0} - \widetilde{\phi}^{\star}\|_{L^{2}(\nu)}^{2}~,
\end{equation*}
where the last step follows from \cref{eqn:first-iterate}.
Since \(t_{N} \geq \frac{N + 1}{2}\), we have
\begin{equation*}
    v_{N} \leq \frac{2 \cdot \lambda(3B) \cdot \|\overline{\phi}^{0} - \phi^{\star}\|_{L^{2}(\nu)}^{2}}{(N + 1)^{2}}~.
\end{equation*}
\end{proof}

\subsubsection{Proof of \cref{prop:path-space-phi-match-update-SDE}}
\label{app:sec:prf:SB-details}

\begin{proof}
We begin by noting that \(\bbP^{n + \nicefrac{1}{2}}\) is the solution to the following SDE
\[\rmd Y_t=[-v_{T-t}^n(Y_t)+2 \nabla\log p_{T-t}^n(Y_t)]\rmd t+\sqrt{2} \cdot \rmd B_{t},\quad Y_0\sim \widetilde{\bbP}^{n}\]
which corresponds to the time-reversal of the SDE that defines \(\bbP^{n}\) and with initial condition given by \(\widetilde{\bbP}^{n}\) whose density is proportional to \(\mathbb{P}^n_T \cdot \frac{\Phi(\nu)}{\Phi(\mathbb{P}^n_T)}\).

The second step follows from \citet[Thm. 4.2]{karimi2024sinkhorn}, which permits us to represent \ref{eqn:second_half_IPFP} as 
\begin{align}
\rmd X_t &= \left(\eta \left[ v_t^{n}(X_t)-2\nabla \log p_t^n(X_t)+2\nabla\log p_t^{n+1/2}(X_t) \right] + (1-\eta) v_t^n(X_t) -2\nabla V_t(X_t) \right) \rmd t \nonumber\\
&\qquad + \sqrt{2} \cdot \rmd B_{t}~,\nonumber\\
% & =[-v_t^{n+1/2}(X_t)+\eta\sigma^2\nabla\log p_t^{n+1/2}(X_t)-\sigma^2\nabla V_t(X_t)] dt+\sigma dW_t \nonumber\\
&= [v_t^{n}(X_t)-2\eta\nabla \log p_t^n(X_t)+2\eta\nabla\log p_t^{n+1/2}(X_t)-2\nabla V_t(X_t)] \rmd t+\sqrt{2} \cdot \rmd B_{t}\label{eqn:recursive_sde_IPF}
\end{align}
where \(X_0\sim \mu\).
The extra drift \(V_{t}\) is as defined in the statement of the proposition.
\end{proof}

\subsubsection{Proof of \cref{lem:SB_dual_formal}}
\label{subsec:proof:SB_dual_formal}
\begin{proof}
The proof of this statement is based on two key equivalences.
Recall that \(\phi_{t}^{n} = \log g_{t}^{n}\) for all \(n \geq 0\) and \(t \in [0, T]\).
These equivalences are:
\begin{enumerate}[leftmargin=*]
    \item between the dual potential $\phi^n$ from eOT and  backward dynamics on $\{\phi^n_t\}_t$ determined by the reference transition: this follows from \citet{caluya2021wasserstein,leonard2014survey} in the case of nonlinear drift (i.e., $u^{\textsf{ref}}\neq 0$).
    \item between the updates on the drifts of the SDE $v_t^n=u^{\textsf{ref}}+\nabla \phi^n_t$ and the path measures $\mathbb{P}^n$ factorized as \cref{eqn:factorize_SB_path}: this follows from classical result on Doob's $h$-transform, which implies that the optimal additional drift for $\mathbb{P}^\star$ should be in the form of $(\nabla\log g_t^\star)_{t\in [0,T]}$ built from the optimal potential $g_T^\star=e^{\phi^{\star}}$.
    The fact that \cref{eqn:factorize_SB_path} is the same as the law of the SDE \cref{eqn:sde_twisted} is also a consequence of the same twisted kernel argument \citep{dai1991stochastic}.
\end{enumerate}

The claim about convergence is primarily due to the factorisation of the path measure.
Since the bridges for these path measures satisfy $\mathbb{P}^n(X_{t\in(0,T)}\vert X_0,X_T)=\mathbb{P}^{\textsf{ref}}(X_{t\in(0,T)}\vert X_0,X_T)$, we have that
$\KLdist(\mathbb{P}^n \|\mathbb{P}^\star) = \KLdist(\pi^n \| \pi^\star)$.
Additionally, due to constraint that $\mathbb{P}_0^n=\mu$ in \cref{eqn:sde_twisted}, this rate is determined by $\mathbb{P}_T^n\rightarrow\nu$ (or equivalently $\phi_T^n=\phi^n\rightarrow \phi^\star$), similar to the static two-marginal case for \ref{eqn:phi-match-update}.
Notably, we also maintain the coupling $\pi^n$ as \(\pi(\phi^{n}, (\phi^{n})^{+})\) which ensures that they satisfy the form of the optimal coupling in \cref{eqn:eOT-opt-form}.
\end{proof}

\subsubsection{Proofs of intermediate lemmas instantiated in this subsection}
\label{sec:prf:intermediate-extensions}

\begin{proof}[Proof of \cref{lem:lower-bound-update}]
Note that \(\sfM^{\textsf{proj-SGA}}_{\calS_{B}}(\phi; \eta) \in \calS_{B}\), and therefore both \(\phi\) and \(\sfM^{\textsf{proj-SGA}}_{\calS_{B}}(\phi; \eta)\) are contained in \(\calS_{\max\{B, \bar{B}\}}\).
By \cref{lem:L2nu-smooth},
\begin{align*}
    \overline{J}_{\calS_{B}}(\sfM^{\textsf{proj-SGA}}_{\calS_{B}}(\phi; \eta)) &= J(\sfM^{\textsf{proj-SGA}}_{\calS_{B}}(\phi; \eta)) - \bbI_{\calS_{B}}(\sfM^{\textsf{proj-SGA}}_{\calS_{B}}(\phi; \eta)) \\
    &\geq J(\phi) + \left\langle \frac{\delta J(\phi)}{\nu}, \sfM^{\textsf{proj-SGA}}_{\calS_{B}}(\phi; \eta) - \phi\right\rangle_{L^{2}(\nu)} \\
    &\qquad - \frac{\lambda(\max\{B, \bar{B}\})}{2} \cdot \|\sfM^{\textsf{proj-SGA}}_{\calS_{B}}(\phi; \eta) - \phi\|_{\nu}^{2} - \bbI_{\calS_{B}}(\sfM^{\textsf{proj-SGA}}_{\calS_{B}}(\phi; \eta)) \\
    &= \widetilde{J}_{\eta, \calS_{B}}(\sfM^{\textsf{proj-SGA}}_{\calS_{B}}(\phi; \eta); \phi) \\
    &\qquad + \left(\frac{1}{2\eta} - \frac{\lambda(\max\{\bar{B}, B\})}{2}\right) \cdot \|\sfM^{\textsf{proj-SGA}}_{\calS_{B}}(\phi; \eta) - \phi\|_{L^{2}(\nu)}^{2} \\
    &\geq \widetilde{J}_{\eta, \calS_{B}}(\sfM^{\textsf{proj-SGA}}_{\calS_{B}}(\phi; \eta); \phi)~.
\end{align*}
The final inequality is due to the fact that \(\eta^{-1} \geq \lambda(\max\{\bar{B}, B\})\).
\end{proof}

\begin{proof}[Proof of \cref{lem:lem2.3-bt09-general}]
The proof generally follows \citet[Lem. 2.3]{beck2009fast} and we include details here for brevity.
By optimality, note that for any \(\phi \in L^{2}(\nu)\)
\begin{equation}
\label{eqn:opt-cond-Q}
    \delta J(\phi) - \frac{1}{\eta} \cdot \nu \cdot (\sfM^{\textsf{proj-SGA}}_{\calS_{B}}(\phi; \eta) - \phi) - \gamma(\phi) = 0; \qquad \gamma(\phi) \in \partial \bbI_{\calS_{B}}(\sfM^{\textsf{proj-SGA}}_{\calS_{B}}(\phi; \eta))~.
\end{equation}

From \cref{lem:lower-bound-update}, we know that
\begin{equation*}
    \overline{J}_{\calS_{B}}(\sfM^{\textsf{proj-SGA}}_{\calS_{B}}(\phi; \eta)) \geq \widetilde{J}_{\eta, \calS_{B}}(\sfM^{\emph{\textsf{proj-SGA}}}_{\calS_{B}}(\phi; \eta); \phi)~.
\end{equation*}
Consequently,
\begin{align*}
    \overline{J}_{\calS_{B}}(\sfM^{\textsf{proj-SGA}}_{\calS_{B}}(\phi; \eta)) - \overline{J}_{\calS_{B}}(\overline{\phi}) &\geq  \widetilde{J}_{\eta, \calS_{B}}(\sfM^{\textsf{proj-SGA}}_{\calS_{B}}(\phi; \eta); \phi) - \overline{J}_{\calS_{B}}(\overline{\phi}) \\
    &= \widetilde{J}_{\eta, \calS_{B}}(\sfM^{\textsf{proj-SGA}}_{\calS_{B}}(\phi; \eta); \phi) - J(\overline{\phi}) + \bbI_{\calS_{B}}(\overline{\phi}) \\
    &\overset{(a)}= J(\phi) + \left\langle \frac{\delta J(\phi)}{\nu}, \sfM^{\textsf{proj-SGA}}_{\calS_{B}}(\phi; \eta) - \phi\right\rangle_{L^{2}(\nu)} - J(\overline{\phi}) \\
    &\quad - \bbI_{\calS_{B}}(\sfM^{\textsf{proj-SGA}}_{\calS_{B}}(\phi; \eta)) - \frac{1}{2\eta} \|\sfM^{\textsf{proj-SGA}}_{\calS_{B}}(\phi; \eta) - \phi\|_{\nu}^{2} + \bbI_{\calS_{B}}(\overline{\phi}) \\
    &\overset{(b)}\geq \left\langle \frac{\delta J(\phi)}{\nu}, \sfM^{\textsf{proj-SGA}}_{\calS_{B}}(\phi; \eta) - \overline{\phi}\right\rangle_{L^{2}(\nu)} + \langle \gamma(\phi), \overline{\phi} - \sfM^{\textsf{proj-SGA}}_{\calS_{B}}(\phi; \eta)\rangle \\
    &\qquad - \frac{1}{2\eta} \|\sfM^{\textsf{proj-SGA}}_{\calS_{B}}(\phi; \eta) - \phi\|_{L^{2}(\nu)}^{2} \\
    &= \left\langle \frac{\delta J(\phi)}{\nu} - \frac{\gamma(\phi)}{\nu}, \sfM^{\textsf{proj-SGA}}_{\calS_{B}}(\phi; \eta) - \overline{\phi}\right\rangle_{L^{2}(\nu)} \\
    &\qquad - \frac{1}{2\eta} \|\sfM^{\textsf{proj-SGA}}_{\calS_{B}}(\phi; \eta) - \phi\|_{L^{2}(\nu)}^{2} \\
    &\overset{(c)}= \frac{1}{\eta} \cdot \langle \sfM^{\textsf{proj-SGA}}_{\calS_{B}}(\phi; \eta) - \phi, \sfM^{\textsf{proj-SGA}}_{\calS_{B}}(\phi; \eta) - \overline{\phi}\rangle_{L^{2}(\nu)} \\
    &\qquad - \frac{1}{2\eta} \|\sfM^{\textsf{proj-SGA}}_{\calS_{B}}(\phi; \eta) - \phi\|_{L^{2}(\nu)}^{2} \\
    &= \frac{1}{2\eta} \cdot \|\sfM^{\textsf{proj-SGA}}_{\calS_{B}}(\phi; \eta) - \phi\|_{L^{2}(\nu)}^{2} \\
    &\qquad + \frac{1}{\eta} \cdot \langle \sfM^{\textsf{proj-SGA}}_{\calS_{B}}(\phi; \eta) - \phi, \phi - \overline{\phi}\rangle_{L^{2}(\nu)}~.
\end{align*}
Step \((a)\) uses the definition of \(\widetilde{J}_{\eta, \calS_{B}}\), step \((b)\) uses the concavity of \(J\) and the convexity of \(\bbI_{\calS_{B}}\) as
\begin{gather*}
    -J(\overline{\phi}) + J(\phi) \geq \langle \delta J(\phi), \phi - \overline{\phi}\rangle~,\\
    \bbI_{\calS_{B}}(\overline{\phi}) - \bbI_{\calS_{B}}(\sfM^{\textsf{proj-SGA}}_{\calS_{B}}(\phi; \eta)) \geq \langle \gamma(\phi), \overline{\phi} - \sfM^{\textsf{proj-SGA}}_{\calS_{B}}(\phi; \eta)\rangle~.
\end{gather*}
Finally, we use the optimality condition (\cref{eqn:opt-cond-Q}) in step \((c)\).
\end{proof}

\begin{proof}[Proof of \cref{lem:key-identity}]
First, since \(\overline{\phi}^{n} \in \calS_{B}\) for all \(n \geq 1\) and \(\frac{t_{n} - 1}{t_{n + 1}} \leq 1\), by the triangle inequality for the semi-norm \(L^{\infty}(\calY)\), we have that \(\phi^{n} \in \calS_{3B}\) for all \(n \geq 0\).
Now, we apply \cref{lem:lem2.3-bt09-general} to two settings.
First, with \(\phi \leftarrow \phi^{n + 1}, \overline{\phi} \leftarrow \overline{\phi}^{n}, \bar{B} \leftarrow 3B\), we have
\begin{align*}
    \overline{J}(\overline{\phi}^{n + 1}) - \overline{J}(\overline{\phi}^{n}) &\geq \frac{\lambda(3B)}{2} \cdot \|\overline{\phi}^{n + 1} - \phi^{n + 1}\|_{L^{2}(\nu)}^{2} + \lambda(3B) \cdot \langle \overline{\phi}^{n + 1} - \phi^{n + 1}, \phi^{n + 1} - \overline{\phi}^{n}\rangle_{L^{2}(\nu)}~.
\end{align*}
Second, with \(\phi \leftarrow \phi^{n + 1}, \overline{\phi} \leftarrow \widetilde{\phi}^{\star}, \bar{B} \leftarrow 3B\), we have
\begin{align*}
    \overline{J}(\overline{\phi}^{n + 1}) - \overline{J}(\widetilde{\phi}^{\star}) &\geq \frac{\lambda(3B)}{2} \cdot \|\overline{\phi}^{n + 1} - \phi^{n + 1}\|_{L^{2}(\nu)}^{2} + \lambda(3B) \cdot \langle \overline{\phi}^{n + 1} - \phi^{n + 1}, \phi^{n + 1} - \widetilde{\phi}^{\star}\rangle_{L^{2}(\nu)}~.
\end{align*}
With the definition of \(v_{k}\), the left hand sides of both inequalities are \(v_{k} - v_{k + 1}\) and \(-v_{k + 1}\) respectively.
The remainder of the proof follows from the proof of \cite[Lem. 4.1]{beck2009fast}.
\end{proof}

\begin{lemma}
\label{lem:tn-recursion}
Consider the recursion 
\begin{equation*}
    t_{k + 1} = \frac{1 + \sqrt{1 + 4 t_{k}^{2}}}{2} \qquad k \geq 1~.
\end{equation*}
If \(t_{1} \geq 1\), then \(0 \leq \frac{t_{k} - 1}{t_{k + 1}} \leq 1\).
\end{lemma}
\begin{proof}
Note that for any \(t_{k}\), \(\frac{1 + \sqrt{1 + 4t_{k}^{2}}}{2} \geq \frac{1 + 1}{2} = 1\)~.
Hence \(\frac{t_{k} - 1}{t_{k + 1}} \geq 0\).
Algebraically,
\begin{align*}
    t_{k} - 1 \leq t_{k + 1} &\Leftrightarrow t_{k} \leq t_{k + 1} + 1 \\
    &\Leftrightarrow t_{k} - \frac{3}{2} \leq \frac{\sqrt{1 + 4t_{k}^{2}}}{2} \\
    &\Leftrightarrow 4t_{k}^{2} + 9 - 12t_{k} \leq 1 + 4t_{k}^{2} \\
    &\Leftrightarrow \frac{2}{3} \leq t_{k}~.
\end{align*}
Since we know that \(t_{k} \geq 1 \geq \frac{2}{3}\), we have \(\frac{t_{k} - 1}{t_{k + 1}} \leq 1\).
\end{proof}

\section{Accelerating SGA through a mirror flow}
\label{sec:md-sga-acceleration}
The probability density w.r.t. the Lebesgue measure of a measure \(\rho\) will be identified by \(\rho\) as well.
We also use \(f, g\) in lieu of \(-\psi, \phi\) for the \(\calX\), \(\calY\) potentials. For convenience, we use the shorthand \(f \oplus g\) to denote \((f \oplus g)(x, y) = f(x) + g(y)\) where \(f : \calX \to \bbR\) and \(g : \calY \to \bbR\).
By direct calculation, we see that \(\delta \varphi^{\star}(f \oplus g) \in \calQ\):
\begin{align*}
    \delta \varphi^{\star}(f \oplus g)(x, y) &= \rmd\mu(x) \cdot \exp\left(g(y) - \frac{c(x, y)}{\varepsilon}\right) \rmd\nu(y) \cdot \left(\int_{\calY} \exp\left(g(y) - \frac{c(x, y)}{\varepsilon}\right)\rmd\nu(y)\right)^{-1} \\
    &= \exp\left(g(y) - g^{+}(x) - \frac{c(x, y)}{\varepsilon}\right)\rmd\mu(x)\rmd\nu(y)~\numberthis\label{eqn:invar-f-oplus}.
\end{align*}

If \(h^{0}\) is of the form \(h^{0} = f^{0} \oplus g^{0}\), then the continuous-time analogue of the update in \cref{eqn:discrete_MD} is given by
\begin{equation*}
\dot{h}^t = \bm{0} \oplus (-(\hat{\pi}_{\calY}^t-\nu)), \quad \hat{\pi}^t=\delta\varphi^{\star}(h^t)\in \calQ\, .
\end{equation*}
This is also equivalent to the ODE
\begin{equation}
    \dot{g}^t =-(\hat{\pi}_{\calY}^t-\nu)~, \quad \hat{\pi}^{t} = \delta \varphi^{\star}(h^{t}) \in \calQ~
\end{equation}
with $\dot{f}^t$ fixed by the dynamics on $\dot{g}^t$ by the $\hat{\pi}_{\calX}^t=\mu$ constraint.
The formal equivalence is established in \cite[Theorem 3.1]{karimi2024sinkhorn}, which also results in the rate
\begin{equation}
\label{eqn:MD_ode_rate}
\KLdist(\hat{\pi}^t_y \Vert \nu)\leq \frac{\KLdist(\pi(g^0,(g^{0})^{+}),\pi^\star)}{t}~.
\end{equation}

For building momentum into the dynamics, we build on \cite{krichene2015accelerated}, together with our interpretation of \ref{eqn:phi-match-update} in \cref{lem:md_interpret}.
Let $r\geq 2$ be a constant.
Define the ODE system below
\begin{subequations}
\begin{align}
\dot{\hat{\pi}}^t &=\frac{r}{t}(\delta\varphi^{\star}(h^t)-\hat{\pi}^t)~,
\label{eqn:ode-1}\\
\dot{h}^t &= \bm{0} \oplus \left(-\frac{t}{r} \cdot (\hat{\pi}_{\calY}^t-\nu)\right)~.\label{eqn:ode-2}
\end{align}
\end{subequations}
The initial conditions to this system is $\hat{\pi}^0 = \delta\varphi^*(h^0(x,y))$ where \(h^{0} = f^{0} \oplus g^{0}\) and hence
\begin{align}
\label{eqn:initial_ode}
\hat{\pi}^0(x, y) &= \exp\left(-\frac{c(x, y)}{\varepsilon}\right)\mu(x)\nu(y) \exp((g^0)^+(x)+g^0(y))\in \calQ~.
\end{align}
We reduce the pair of ODEs to a single ODE.
This is done by rewriting \cref{eqn:ode-1} as
\[t^r \dot{\hat{\pi}}^t+rt^{r-1}\hat{\pi}^t=rt^{r-1}\delta\varphi^*(h^t)\]
and upon time integration,
\begin{equation}
\label{eqn:weighted_sum_pi}
t^r\hat{\pi}^t = r\int_0^t \tau^{r-1} \delta\varphi^*(h^\tau) \rmd\tau\; \Leftrightarrow\; \hat{\pi}^t=\frac{\int_0^t \tau^{r-1} \delta\varphi^*(h^\tau) \rmd\tau}{\int_0^t \tau^{r-1} \rmd\tau}~.
\end{equation}
\cref{eqn:weighted_sum_pi} implies that the solution \(\hat{\pi}^{t}\) at each \(t\) is weighted ``sum'' of \(\left(\delta \varphi^{\star}(h^{\tau})\right)_{\tau \in (0, t)}\) that is based on \cref{eqn:ode-2}.
With the initial condition as in \(h^{0} = f^{0} \oplus g^{0}\), the second update \cref{eqn:ode-2} says the dual ``variable'' \(g^{t}\) accumulates gradient at a certain rate.
More precisely, \cref{eqn:ode-2} can be equivalently written as
\begin{equation}
\label{eqn:second_update}
    \dot{g}^{t} = -\frac{t}{r} \cdot (\hat{\pi}^{t}_{\calY} - \nu) \qquad h^{t} = f^{0} \oplus g^{t}~.
\end{equation}

Hence, to implement this, we only require access to \(\hat{\pi}^{t}_{\calY}\), which can be computed from \cref{eqn:weighted_sum_pi} as
\begin{align}
\hat{\pi}_{\calY}^t&=\int_{\calX} \frac{\int_0^t\tau^{r-1}\delta\varphi^{\star}(h^{\tau})\rmd\tau}{\int_0^t \tau^{r-1} \rmd\tau} \rmd x\nonumber \\
&=\left(\int_0^t \tau^{r-1} \rmd\tau\right)^{-1} \cdot \int_0^t\tau^{r-1} \left\{\int_{\calX} \exp\left(g^\tau(y)-(g^{\tau})^{+}(x) - \frac{c(x, y)}{\varepsilon}\right) \rmd\mu(x)\rmd\nu(y)\right\} \rmd\tau\label{eqn:actual_update}
\end{align}
where we used \(\delta\varphi^{\star}(h^{\tau}) \in \calQ\) from \cref{eqn:invar-f-oplus}. Note that \cref{eqn:actual_update} is only a function of $g$.
Consequently, the original accelerated mirror descent dynamics in \cref{eqn:ode-1,eqn:ode-2} are reduced to \cref{eqn:second_update}-\cref{eqn:actual_update} in terms of $g^t, \hat{\pi}_\calY^t$ only.
However, the non-convexity of \(\calQ\) doesn't guarantee that \(\hat{\pi}^{t} \in \calQ\) despite being a convex combination of \(\rho^{\tau}:= \delta\varphi^{\star}(h^{\tau})\) which individually lie in \(\calQ\).
This makes it somewhat difficult to argue if \(\hat{\pi}^{t}\) is approaching an optimal coupling.

However, this doesn't preclude us from giving a rate of convergence for the continuous-time dynamics \cref{eqn:ode-1,eqn:ode-2} for the \(\calY\)-marginal.

\begin{lemma}
\label{lem:acc-MD-rate}
Let \((\hat{\pi}^{t})_{t \geq 0}\) be the solution to the system \cref{eqn:ode-1,eqn:ode-2}.
For kernel \(k(y, y') = 1\) iff \(y = y'\), it holds for any \(r \geq 2\) that
\begin{equation*}
    \calL_{k}(\hat{\pi}^{t}_{\calY}, \nu) \leq \frac{r^{2}}{t^{2}} \cdot \KLdist(\pi(g^{0}, (g^{0})^{+}), \pi^{\star})~,
\end{equation*}
where $\calL_{k}(\cdot; \nu) = \frac{1}{2}\mathrm{MMD}_{k}(\cdot, \nu)^{2}$.
\end{lemma}

\begin{proof}
The proof is based on showing the decay of the following Lyapunov functional
\begin{equation*}
V(\hat{\pi}_y^t,g^t,t)=\frac{t^2}{r} \calL_{k}(\hat{\pi}_y^t;\nu)+rD_{\varphi^{\star}}((f^{0} \oplus g^t); (f^{0} \oplus g^{\star}))~,
\end{equation*}
where \(D_{\varphi^{\star}}(h'; h) = \varphi^{\star}(h') - \varphi^{\star}(h) - \left\langle\delta \varphi^{\star}(h), h' - h\right\rangle\).
The time derivative of \(V\) is
\begin{align*}
\frac{\rmd}{\rmd t}V(\hat{\pi}_y^t,g^t,t)&=\frac{2t}{r}\calL_{k}(\hat{\pi}_y^t; \nu) + \frac{t^2}{r}\langle \delta \calL_{k}(\hat{\pi}_y^t; \nu),\dot{\hat{\pi}}^t\rangle\\
&\qquad +r\langle \dot{g}^t,\nabla\varphi^{\star}((f^{0} \oplus g^t))-\nabla\varphi^{\star}((f^{0} \oplus g^{\star}))\rangle\\
&= \frac{2t}{r}\calL_{k}(\hat{\pi}_y^t; \nu) + t\langle \delta \calL_{k}(\hat{\pi}_y^t; \nu), \frac{t}{r}\dot{\hat{\pi}}^t-\nabla\varphi^{\star}((f^{0} \oplus g^t))+\nabla\varphi^{\star}((f^{0} \oplus g^{\star}))\rangle\\
&= \frac{2t}{r}\calL_{k}(\hat{\pi}_y^t; \nu) - t\langle \delta \calL_{k}(\hat{\pi}_y^t;\nu), \hat{\pi}^t-\pi^{\star}\rangle\\
&= -\frac{t(r-2)}{r} \calL_{k}(\hat{\pi}_y^t; \nu)\, .
\end{align*}
Therefore if $r\geq 2$, $V$ is decreasing with time, and we have
\begin{align*}
    \frac{t^2}{r} \calL_{k}(\hat{\pi}_y^t; \nu) &\leq V(\hat{\pi}_y^t,g^t,t) \\
    &\leq V(\hat{\pi}_y^0,g^0,0) \\
    &=rD_{\varphi^{\star}}((f^{0} \oplus g^t),(f^{0} \oplus g^{\star}))\\ 
    &=r \cdot \KLdist(\pi(g^0,(g^{0})^{+}),\pi^{\star})~.
\end{align*}
where in the last step, we used \citet[Lem. B.1]{karimi2024sinkhorn}.
This completes the proof.
\end{proof}

Although the rate above is $\mathcal{O}(1/t^2)$, faster than MD ODE with $\mathcal{O}(1/t)$ rate in \cref{eqn:MD_ode_rate}, the lemma does not establish convergence of $\hat{\pi}^t\rightarrow \pi^*$, only the convergence for the $\calX$ and the $\calY$ marginals. However, from \cref{eqn:actual_update} we have that $\hat{\pi}^t$ is a weighted sum of $\rho^\tau\in \calQ$ and since the weights increase very fast with $\tau$ after running long enough we'd expect $\hat{\pi}^t\approx \rho^t\in \calQ$ for large $t$.
What we can conclude from this discussion is that despite the niceness of \(\calL_{k}\) from a mirror descent standpoint (e.g., smooth and convex), the non-convexity of the constraint set we are optimizing over poses a fundamental challenge for acceleration, which involves interpolation between past iterates.
This also highlights the benefit of the semi-dual perspective for building momentum, as we do in \cref{sec:extensions:steep-ascent}.

\section{Semi-dual problem in OT}
\label{sec:semi-dual-OT}
For completeness, here we briefly highlight the semi-dual problem associated with the OT problem (\cref{eqn:OT-prob-primal}) as well.
The Monge-Kantorovich dual problem in OT \citep[Chap. 1]{villani2003topics} is given by
\begin{equation}
\label{eqn:OT-prob-dual}
\opttrans(\mu, \nu; c) = \sup \left\{ \int f(x) d\mu(x) + \int g(y) \rmd\nu(y) : f(x) + g(y) \leq c(x, y)\right\}~.
\end{equation}

The \(c\)-transform of a continuous function \(\varphi\) is defined as
\begin{equation*}
    \varphi^{c}(x) = \inf_{y} \; c(x, y) - \varphi(y)~.
\end{equation*}
From \cite{gangbo1996geometry}, any solution \((f^{\star}, g^{\star})\) that solves \cref{eqn:OT-prob-dual} is a pair of functions that satisfies
\begin{equation*}
    f^{\star} = (g^{\star})^{c} \qquad g^{\star} = (f^{\star})^{c}~.
\end{equation*}
Consequently, the dual problem (\cref{eqn:OT-prob-dual}) can be reduced to a problem over just one of \(f \in L^{1}(\mu)\) or \(g \in L^{1}(\nu)\) \citep[Thm. 1]{jacobs2020fast} as
\begin{equation}
\label{eqn:semi-dual-OT}
    \opttrans(\mu, \nu; c) = \sup_{f \in L^{1}(\mu)} \int f(x)\rmd \mu(x) + \int f^{c}(y)\rmd \nu(y)~.
\end{equation}

\paragraph{Semi-dual gradient ascent for OT}
We would like to highlight here that iterative methods for maximising the semi-dual for the OT problem have also been previously proposed \citep{chartrand2009gradient,jacobs2020fast}.
\citet{chartrand2009gradient} consider the \(L^{2}\)-gradient (equivalently, the Fr\'{e}chet derivative) of the OT semi-dual, while recent work by \citet{jacobs2020fast} use the \(\dot{H}^{1}\)-gradient defined with respect to the homogeneous Sobolev space \(\dot{H}^{1}\) which \citet{jacobs2020fast} argue is better suited to the OT problem.

\section{Additional remarks on \ref{eqn:kernel-sga-update}}
\subsection{Reparameterisation}
Alternatively, one can also interpret \ref{eqn:kernel-sga-update} as gradient ascent on a re-parametrised dual objective.
Let \(\widetilde{J} : L^{1}(\nu) \to \bbR\) be defined as
\begin{equation*}
    \widetilde{J}(\varphi) = J\left(\int_{\calY} k(z, \cdot) \varphi(z)\rmd z\right)~.
\end{equation*}
Let \(M_{k} : L^{1}(\nu) \to L^{\infty}(\calY)\) be defined as \(M_{k}(g) = \int_{\calY} k(z, \cdot)g(z)\rmd z\), and note that this is linear.
Therefore,
\begin{align*}
    \delta \widetilde{J}(\varphi)(y) &= \delta J(M_{k}(\varphi))[\delta M_{k}(\varphi)(y)] \\
    &= \int_{\calY} \delta J(M_{k}(\varphi))(z) \cdot k(z, y) \rmd z ~.
\end{align*}
A gradient ascent procedure for \(\widetilde{J}\) results in the iteration
\begin{align*}
    \varphi^{n + 1} &= \varphi^{n} + \eta \cdot \delta \widetilde{J}(\varphi^{n}) \numberthis\label{eqn:grad-ascent-kernelised-J} \\
    \Rightarrow M_{k}(\varphi^{n + 1}) &= M_{k}\left(\varphi^{n} + \eta \cdot \delta \widetilde{J}(\varphi^{n}) \right) \\
    &= M_{k}(\varphi^{n}) + \eta \cdot \int_{\calY} k(w, \cdot) \cdot \delta \widetilde{J}(\varphi^{n})(w) \rmd w \\
    &= M_{k}(\varphi^{n}) + \eta \cdot \int_{\calY} k(w, \cdot) \int_{\calY} \delta J(M_{k}(\varphi^{n}))(z) \cdot k(w, z) \rmd z \rmd w \\
    &= M_{k}(\varphi^{n}) + \eta \cdot \int_{\calY} \delta J(M_{k}(\varphi^{n}))(z) \cdot k(z, \cdot)\rmd z~.
\end{align*}
This shows that a sequence of \(\{\varphi^{n}\}_{n \geq 0}\) obtained through \cref{eqn:grad-ascent-kernelised-J} can be mapped to \(\{\phi^{n}\}_{n \geq 0}\) obtained from \ref{eqn:kernel-sga-update} as \(\phi^{n} = M_{k}(\varphi^{n})\) for all \(n \geq 0\).

\subsection{Particle implementation of \(k\)\textsf{-SGA}}

When \(k\) is a Gaussian or a Laplace kernel, we can use the form of \(\frakm_{k}\) and independent samples from a Gaussian or Laplace distribution respectively to implement \ref{eqn:kernel-sga-update} approximately.
This is because \(\frakm_{k}(f)(y) = C_{k} \cdot \bbE_{\xi \sim \calD_{k}}\left[f(y + \xi)\right]\) where the pair \((C_{k}, \calD_{k})\) are
\begin{equation*}
    (C_{k},\calD_{k}) =
    \begin{cases}
        (2\pi \sigma^{2})^{\nicefrac{d}{2}},~\calN(0, \sigma^{2} \cdot \rmI_{d}) & \text{if } k(y, y') = \exp\left(-\frac{\|y' -y\|^{2}_{2}}{2\sigma^{2}}\right) \\
        (2b)^{d},~\mathrm{Lap}(0, b) & \text{if } k(y, y') = \exp\left(-\frac{\|y' - y\|_{1}}{2b}\right)
    \end{cases}~.
\end{equation*}
As a result, the particle version of \ref{eqn:kernel-sga-update} can be expressed as
\begin{equation*}
    \widehat{\sfM}^{k\textsf{-SGA}}(\phi; \eta) = \phi + \frac{\eta \cdot C_{k}}{N_{\text{samp}}} \cdot \left\{\sum_{i=1}^{N_{\text{samp}}} \nu(y + \xi_{i}) - \pi(\phi, \phi^{+})_{\calY}(y + \xi'_{i})\right\}~; \quad \{\xi_{i}\}, \{\xi_{i}'\} \overset{\mathrm{i.i.d.}}{\sim} \calD_{k}~.
\end{equation*}

Besides these, for a general \(k\), it might be viable to consider the Stein kernel \citep{liu2016kernelized}.
The Stein kernel for a distribution \(\rho\) and with a base kernel \(k\) is given by
\begin{align*}
    k_{\rho}(x, x') &= k(x, x') \cdot \langle \nabla \log \rho(x), \nabla \log \rho(x')\rangle + \langle \nabla \log \rho(x), \nabla_{2}k(x, x')\rangle \\
    &\qquad + \langle \nabla \log \rho(x'), \nabla_{1}k(x, x')\rangle + \nabla_{1} \cdot (\nabla_{2}k(x, x'))~.
\end{align*}
The key property of the Stein kernel is that
\begin{equation*}
    \int k_{\rho}(z, \cdot) \rho(z) \rmd z = 0~.
\end{equation*}
When given an oracle to sample from \(\pi(\phi, \phi^{+})_\calY\) but not \(\nu\), it would be useful to consider the Stein kernel \(k_{\nu}\), and with this choice of \(k \leftarrow k_{\nu}\), \ref{eqn:kernel-sga-update} becomes
\begin{equation*}
    \phi^{n + 1} = \phi^{n} - \eta \cdot \int_{\calY} k_{\nu}(z, \cdot) \pi(\phi^{n}, (\phi^{n})^{+})_{\calY}(z)\rmd z~.
\end{equation*}

\end{document}